 \documentclass[12pt]{amsart}

\usepackage[all]{xy}
\usepackage{amssymb}

\allowdisplaybreaks[3]

\def\<{\langle}
\def\>{\rangle}
\let\ipscriptstyle=\scriptscriptstyle
\def\lipsqueeze{{\mskip -3.0mu}}
\def\ripsqueeze{{\mskip -3.0mu}}
\def\ipcomma{\nobreak\mathrel{,}\nobreak}
\newbox\ipstrutbox
\setbox\ipstrutbox=\hbox{\vrule height8.5pt
width 0pt}
\def\ipstrut{\copy\ipstrutbox}
\def\lip#1<#2,#3>{\mathopen{\relax_{\ipstrut\ipscriptstyle{
#1}}\lipsqueeze
\langle} #2\ipcomma #3 \rangle}
\def\blip#1<#2,#3>{\mathopen{\relax_{\ipstrut
\ipscriptstyle{ #1}}\lipsqueeze\bigl\langle} #2\ipcomma #3 \bigr\rangle}
\def\rip#1<#2,#3>{\langle #2\ipcomma #3
\rangle_{\ripsqueeze\ipstrut\ipscriptstyle{#1}}}
\def\brip#1<#2,#3>{\bigl\langle #2\ipcomma #3
\bigr\rangle_{\ripsqueeze\ipstrut\ipscriptstyle{#1}}}
\def\angsqueeze{\mskip -6mu}
\def\smangsqueeze{\mskip -3.7mu}
\def\trip#1<#2,#3>{\langle\smangsqueeze\langle #2\ipcomma #3
\rangle\smangsqueeze\rangle_{\ripsqueeze\ipstrut\ipscriptstyle{#1}}}
\def\btrip#1<#2,#3>{\bigl\langle\angsqueeze\bigl\langle #2\ipcomma
#3
\bigr\rangle
\angsqueeze\bigr\rangle_{\ripsqueeze\ipstrut\ipscriptstyle{#1}}}
\def\tlip#1<#2,#3>{\mathopen{\relax_{\ipstrut\ipscriptstyle{
#1}}\lipsqueeze \langle\smangsqueeze\langle} #2\ipcomma #3
\rangle\smangsqueeze\rangle}
\def\btlip#1<#2,#3>{\mathopen{\relax_{\ipstrut\ipscriptstyle{
#1}}\lipsqueeze
\bigl\langle\angsqueeze\bigl\langle} #2\ipcomma #3
\bigr\rangle\angsqueeze\bigr\rangle}

\def\ip(#1|#2){(#1\mid #2)}
\def\bip(#1|#2){\bigl(#1 \mid #2\bigr)}
\def\Bip(#1|#2){\Bigl( #1 \bigm| #2 \Bigr)}

\newtheorem{theorem}{Theorem}[section]
\newtheorem{thm}[theorem]{Theorem}
\newtheorem{lemma}[theorem]{Lemma}
\newtheorem{prop}[theorem]{Proposition}
\newtheorem{cor}[theorem]{Corollary}

\theoremstyle{remark}

\newtheorem{remark}[theorem]{Remark}

\newtheorem*{problem*}{Problem}
\newtheorem*{remark*}{Remark}
\newtheorem*{convention*}{Convention}
\newtheorem*{notation*}{Notation}
\newtheorem*{examples*}{Examples}
\newtheorem*{example*}{Example}
\newtheorem*{warning*}{Warning}

\numberwithin{equation}{section}

\def\DD{{\mathcal D}}

\def\H{{\mathcal H}}
\def\K{{\mathcal K}}
\def\L{{\mathcal L}}

\def\Inf{\operatorname{Inf}}
\newcommand{\Rie}{\operatorname{{Rie}}}

\newcommand{\nd}{\textup{nd}}
\newcommand{\Fix}{\operatorname{\mathtt{Fix}}}
\newcommand{\RCP}{\operatorname{\mathtt{RCP}}}
\newcommand{\CP}{\operatorname{\mathtt{CP}}}
\newcommand{\CC}{\operatorname{\mathtt{C*}}}
\newcommand{\CCnd}{\operatorname{\CC_{\!\!\!\!\nd}}}
\newcommand{\Aa}{\operatorname{\mathtt{C*act}}}
\newcommand{\Aca}{\operatorname{\mathtt{C*coact}}}

\newcommand{\nor}{\textup{n}}

\newcommand\tensor{\otimes}

\newcommand{\op}{{\operatorname{op}}}
\newcommand{\lt}{\operatorname{lt}}
\newcommand{\rt}{\textup{rt}}
\newcommand{\id}{\operatorname{id}}
\newcommand{\Ad}{\operatorname{Ad}}
\newcommand{\Ind}{\operatorname{Ind}}
\def\dashind{\operatorname{\!-Ind}}
\newcommand{\Aut}{\operatorname{Aut}}

\newcommand{\supp}{\operatorname{supp}}
\newcommand{\lsp}{\operatorname{span}}

\newcommand{\clsp}{\overline{\operatorname{span}}}
\def\dashind{\operatorname{\!-Ind}}
\def\LHS{\textup{LHS}}
\def\RHS{\textup{RHS}}

\def\Rep{\operatorname{Rep}}
\def\Res{\operatorname{Res}}
\def\dec{\textup{dec}}

\newcommand{\ib}{im\-prim\-i\-tiv\-ity bi\-mod\-u\-le}

\begin{document}

\title[Fixed-point algebras]{Fixed-point algebras for proper
  actions\\ and crossed products by homogeneous spaces}

\author[an Huef]{Astrid an Huef}
\address{School of Mathematics and Statistics\\
The University of New South Wales\\Sydney\\
NSW 2052\\
Australia}
\email{astrid@unsw.edu.au}

\author[Kaliszewski]{S. Kaliszewski}
\address{Department of Mathematical and Statistical Sciences\\Arizona
State University\\Tempe\\ AZ
85287-1804\\USA} \email{kaliszewski@asu.edu}

\author[Raeburn]{Iain Raeburn}
\address{School  of Mathematics and Applied Statistics, University of
Wollongong, NSW 2522, Australia}
\email{raeburn@uow.edu.au}

\author[Williams]{Dana P. Williams}
\address{Department of Mathematics\\Dartmouth College\\
Hanover, NH 03755\\USA}
\email{dana.williams@dartmouth.edu}

\begin{abstract}
We consider a fixed free and proper action of a locally compact group $G$ on a space $T$, and actions $\alpha:G\to \Aut A$ on $C^*$-algebras for which there is an equivariant embedding of $(C_0(T),\rt)$ in $(M(A),\alpha)$. A recent theorem of Rieffel implies that $\alpha$ is proper and saturated with respect to the subalgebra $C_0(T)AC_0(T)$ of $A$, so that his general theory of proper actions gives a Morita equivalence between $A\rtimes_{\alpha,r}G$ and a generalised fixed-point algebra $A^\alpha$. Here we investigate the functor $(A,\alpha)\mapsto A^\alpha$ and the naturality of Rieffel's Morita equivalence, focusing in particular on the relationship between the different functors associated to subgroups and quotients. We then use the results to study  induced representations for crossed products by coactions of homogeneous spaces $G/H$ of $G$, which  were previously shown by an Huef and Raeburn to be fixed-point algebras for the dual action of $H$ on the crossed product by $G$. 
\end{abstract}

\date{July 2, 2009}

\thanks{This research was supported by the Australian Research Council
  and the Edward Shapiro fund at Dartmouth College.}

\maketitle

\section{Introduction} 
Suppose that a locally compact group $G$ acts freely and properly on the right of a locally compact space $T$, and $\rt$ is the induced action of $G$ on $C_0(T)$. Let $\alpha:G\to \Aut A$ be an action of $G$ on a $C^*$-algebra $A$.   Rieffel proved in \cite[Theorem~5.7]{integrable} that if there is a nondegenerate homomorphism $\phi:C_0(T)\to M(A)$ such that $\alpha_s\circ\phi=\phi\circ \rt_s$, then $\alpha$ is proper and saturated in the sense of \cite{proper} with respect to the subalgebra $A_0:=\phi(C_c(T))A\phi(C_c(T))$ of $A$.  The general theory of \cite{proper} then implies that the reduced crossed product $A\rtimes_{\alpha, r} G$ is Morita equivalent via a bimodule $Z(A,\alpha,\phi)$ to a generalised fixed-point algebra $A^\alpha$ sitting in the multiplier algebra $M(A)$ of $A$.

 Theorem~5.7 of \cite{integrable} covers all the main examples of proper actions, such as the dual actions on crossed products by coactions and the actions on graph algebras induced by free actions on graphs  (see \cite[Remark~4.5]{aHRWproper2}), and this has opened up new applications of the theory in \cite{proper}. It was used in \cite{hrman}, for example,  to show that if $\delta$ is a coaction of $G$ on $B$ and $H$ is a closed subgroup of $G$, then the dual action $\hat\delta$ of $H$ on the crossed product $B\rtimes_\delta G$ is proper  and saturated with fixed-point algebra the crossed product $B\rtimes_{\delta}(G/H)$ by the homogeneous space. The resulting Morita equivalence of $(B\rtimes_{\delta} G)\rtimes_{\hat\delta} H$ with $B\rtimes_{\delta}(G/H)$ extends Mansfield's imprimitivity theorem for crossed products by coactions to arbitary closed subgroups (as opposed to the normal subgroups in \cite{man} and \cite{kqman}) \cite[Theorem~3.1]{hrman}.  More generally, the identification of the fixed-point algebra as a crossed product by a homogeneous space promises to give us useful leverage for studying this previously intractable family of crossed products.

The results in \cite{proper} depend on the existence of a dense invariant subalgebra $A_0$, and the Morita equivalences obtained there depend on the choice of $A_0$. Theorem~5.7 of \cite{integrable}, on the other hand, says that in the presence of the homomorphism $\phi$, there is a canonical choice $\phi(C_c(T))A\phi(C_c(T))$ for $A_0$. This makes it possible to ask questions about the functoriality and naturality of the constructions in \cite{proper} on categories of triples $(A,\alpha,\phi)$. Such questions were answered in \cite{kqrproper} for a category which is particularly appropriate to Landstad duality for crossed products by coactions, and the general theory in \cite{kqrproper} has  interesting implications for nonabelian duality for crossed products. The success of this program prompted the present authors to investigate the naturality of Rieffel's Morita equivalence in a category modelled on those in \cite{enchilada}, where the isomorphisms are given by Morita equivalences \cite{aHKRWproper-n}. The main result of \cite{aHKRWproper-n} says that Rieffel's Morita equivalences implement a natural isomorphism between a crossed-product functor and a functor $\Fix$ which sends $(A,\alpha,\phi)$ to $\Fix(A,\alpha,\phi):=A^\alpha$.  This implies in particular that the Morita equivalence of \cite{hrman} is a natural isomorphism in an appropriate category \cite[Theorem~5.6]{aHKRWproper-n}. 

In this paper, we investigate applications of Rieffel's construction to crossed products by coactions of homomogeneous spaces.  Our ultimate goal is to study  and in particular to construct and understand induced representations of crossed products by homogeneous spaces. Here we construct an induction process based on the bimodules implementing Rieffel's equivalence, and investigate its properties. Our main theorem establishes induction-in-stages for induced representations of crossed products by homogeneous spaces. As in \cite{aHKRWproper-n}, we work as far as possible in the ``semi-comma category'' associated to a free and proper action of $G$ on a space $T$, and obtain our results for crossed products by homogeneous spaces by applying the general theory to the pairs $(T,G)=(G,H)$ associated to closed subgroups $H$ of~$G$.

We begin by reviewing the definition and properties of the semi-comma category, Rieffel's proper actions, and the basics of coactions and their crossed products. We also prove a key property of the inclusion of $\Fix(A,\alpha,\phi)$ in $M(A)$ which was used implicitly at a key point in \cite[Theorem~2.6]{kqrproper}. In \S\ref{sec:fixandGreen}, we discuss the relationship between fixed-point algebras and Green induction of representations for crossed products by actions. Theorem~\ref{thm-extension} says that Green induction from $H$ to $G$ is dual to the restriction map associated to an inclusion of fixed-point algebras, which immediately implies an induction-restriction result for crossed products by homogeneous spaces associated to an arbitrary pair of closed subgroups $H\subset K$ (Corollary~\ref{prop-prob2homog}). It also implies that the Green bimodules define a natural transformation between reduced-crossed-product functors (Theorem~\ref{thm-green}).

In \S\ref{sec:abstractfis}, we discuss the general version of induction-in-stages. This has two main ingredients: we need to extend the functor $\Fix$ to an equivariant version, and then prove what we call ``fixing-in-stages":  fixing first over a normal subgroup and then over the quotient is naturally isomorphic to fixing once over the whole group (Theorem~\ref{thm-newfix}). Once we have this, it makes sense to prove that the construction of Rieffel's Morita equivalence can also be done in stages (Theorem~\ref{cor-easyiis}). Next, in \S\ref{sec:plums}, we produce an equivariant version of the natural equivalence of \cite[Theorem~3.5]{aHKRWproper-n}.

The last two sections contain our main applications to crossed products by coactions. In \S\ref{sec:nat}, we use our results on fixing-in-stages to prove a decomposition theorem for crossed products by homogeneous spaces (Theorem~\ref{newMe}). Unfortunately, because we only have the results of \S\ref{sec:abstractfis} for normal subgroups, we need to restrict attention to pairs of subgroups $H\subset K$ for which $H$ is normal in $K$. However, we do not need to assume that $H$ and $K$ are normal in $G$, and hence our results represent a substantial step forward. We stress that, while our Theorem~\ref{newMe} constructs a natural isomorphism, in this case the existence of the isomorphism (which in our category is really  a Morita equivalence) is itself new. In our last section, we discuss induction of representations from one crossed product by a homogeneous space to another, and in particular prove induction-in-stages.

\section{Preliminaries}

\subsection{\boldmath{Categories of $C^*$-algebras}} We are interested in several categories of $C^*$-algebras, and we have tried to stick to the conventions of \cite{aHKRWproper-n}. In the category $\CC$, the objects are $C^*$-algebras and the morphisms from $A$ to $B$ are isomorphism classes of right-Hilbert bimodules ${}_AX_B$ (where, as in \cite{aHKRWproper-n}, we always assume that $X$ is full as a Hilbert $B$-module and that the left action of $A$ is given by a nondegenerate homomorphism $\kappa:A\to \L(X)$). We denote the category used in \cite{kqrproper}, in which the morphisms from $A$ to $B$ are nondegenerate homomorphisms $\sigma:A\to M(B)$, by $\CCnd$. As in \cite{aHKRWproper-n}, we are interested in functors defined on the category $\Aa(G)$, in which the objects are dynamical systems $(A,\alpha)$ consisting of an action $\alpha$ of a fixed locally compact group $G$ on a $C^*$-algebra $A$ and the morphisms are suitably equivariant right-Hilbert bimodules. 

Now suppose that $G$ acts freely and properly on the right of a locally compact space $T$, and $(C_0(T),\rt)$ is the corresponding object in $\Aa(G)$. The objects in both the semi-comma category $\Aa(G, (C_0(T),\rt))$ of \cite{aHKRWproper-n} and the comma category $(C_0(T),\rt)\downarrow\Aa_{\nd}(G)$ of \cite{kqrproper} consist of an object $(A,\alpha)$ in $\Aa(G)$ together with a nondegenerate homomorphism $\phi:C_0(T)\to M(A)$ which is $\rt$--$\alpha$ equivariant. In $\Aa(G, (C_0(T),\rt))$, however, the morphisms from $(A,\alpha,\phi)$ to $(B,\beta,\psi)$ are just the morphisms from $(A,\alpha)$ to $(B,\beta)$ in $\Aa(G)$ (ignoring $\phi$ and $\psi$), whereas in $(C_0(T),\rt)\downarrow\Aa_{\nd}(G)$, they are nondegenerate homomorphisms $\sigma:A\to M(B)$ such that $\psi=\sigma\circ\phi$.  Our reasons for choosing to work in the semi-comma category are explained in \cite[\S2]{aHKRWproper-n}. Of crucial importance, in \cite{enchilada}, in \cite{aHKRWproper-n}, and here, is Proposition~2.1 of \cite{aHKRWproper-n}, which says that every morphism $[{}_AX_B]$ in $\Aa(G, (C_0(T),\rt))$ is the composition of an isomorphism (that is, a morphism coming from an imprimitivity bimodule) and a morphism coming from a nondegenerate homomorphism $\kappa:A\to M(\K(X))=\L(X)$.

\subsection{Fixing} Suppose that $G$ acts freely and properly on $T$. We consider an object $(A,\alpha,\phi)$ in $\Aa(G, (C_0(T),\rt))$, and the subalgebra $A_0:=\lsp\{\phi(f)a\phi(g):a\in A,\ f,g\in C_c(T)\}$ of $A$. As in \cite{aHKRWproper-n,kqrproper}, we simplify the notation by dropping the $\phi$ from our notation when no ambiguity seems possible, and by writing, for example, $XB$ for $\lsp\{x\cdot b:x\in X,\ b\in B\}$. Thus
\[
A_0=C_c(T)AC_c(T)=\lsp\{fag:a\in A,\ f,g\in C_c(T)\}.
\]
It was shown in \cite[\S2]{kqrproper}, using a theory developed by
Olesen-Pedersen \cite{OP1, OP2} and Quigg \cite{quigg, QR}, that for
every $a\in A_0$, there exists $E(a)\in M(A)$ such that
\begin{equation*}
  \omega(E(a))=\int_G \omega(\alpha_s(a))\,ds\ \text{ for $\omega\in A^*$;}
\end{equation*}
we will write $E^G$ or $E^\alpha$ for $E$ if we want to emphasise the particular group or action.
The range $E(A_0)$ of $E$ is a $*$-subalgebra of $M(A)^\alpha$, and $\Fix(A,\alpha,\phi)$ is by definition the norm closure $\overline{E(A_0)}$ in $M(A)$. We showed in \cite[Theorem~3.3]{aHKRWproper-n} that $\Fix$ extends to a functor from the semi-comma category $\Aa(G, (C_0(T),\rt))$ to $\CC$. 

Rieffel's theorem (Theorem~5.7 of \cite{integrable}) implies that the action $\alpha$ is proper and saturated with respect to $A_0$, and so the theory of \cite{proper} gives us a Morita equivalence between the reduced crossed product $A\rtimes_{\alpha,r}G$ and a generalized-fixed point algebra $A^\alpha$ sitting in $M(A)$, implemented by an imprimitivity bimodule  $Z(A,\alpha,\phi)$. It was shown in \cite[Proposition~3.1]{kqrproper} that $\Fix(A,\alpha,\phi)$ coincides with the algebra $A^\alpha$ appearing in \cite{proper}. (We use the notation  $A^\alpha$ for $\Fix(A,\alpha,\phi)$ if we think it is obvious what $\phi$ is, and add a subscipt $\Fix_G$ if there is more than one group around.)  We proved in \cite[Theorem~3.5]{aHKRWproper-n} that the assignments $(A,\alpha,\phi)\mapsto Z(A,\alpha,\phi)$ implement a natural isomorphism between $\Fix$ and a reduced-crossed-product functor $\RCP:\Aa(G, (C_0(T),\rt))\to \CC$.

\subsection{Coactions} The main applications of our theory are to crossed products by coactions, and in particular to crossed products by coactions of homogeneous spaces. We have formulated some of these applications as corollaries to our more general results, so it seems useful to set out our conventions at the start.

As in \cite{aHKRWproper-n}, we work exclusively with the normal nondegenerate coactions introduced in \cite{quigg-fr}  and discussed in \cite[Appendix~A]{enchilada}; we explained in \cite[\S5]{aHKRW-JOT} why normal coactions seem to be particularly appropriate when dealing with homogeneous spaces. We use the category $\Aca^\nor(G)$ in which the objects $(B,\delta)$ are normal coactions $\delta:B\to M(B\otimes C^*(G))$, and the morphisms from $(B,\delta)$ to $(C,\epsilon)$ are isomorphism classes $[X,\Delta]$ of right-Hilbert $B$\,--\,$C$ bimodules $X$ carrying a $\delta$\,--\,$\epsilon$ compatible coaction $\Delta$ (see Theorem~2.15 of \cite{enchilada}). We denote the crossed product of $(B,\delta)$ by $(B\rtimes_\delta G,j_B,j_G)$, and then Theorem~3.13 of \cite{enchilada} says that the assignments
\[
(B,\delta)\mapsto B\rtimes_\delta G\quad\text{and}\quad[X,\Delta]\mapsto [X\rtimes_\Delta G]
\]
form a functor $\CP:\Aca^\nor(G)\to \CC$.  

The dual action $\hat\delta: G\to\Aut (B\rtimes_\delta G)$ is characterised by
\[
\hat\delta_t(j_B(b)j_G(f))=j_G(b)j_G(\rt_t(f))\ \text{ for $b\in B$ and $f\in C_0(G)$,}
\]
and $(B\rtimes_\delta G,\hat\delta,j_G)$ is an object in our semi-comma category $\Aa(G, (C_0(G),\rt))$. Now let $H$ be a closed subgroup of $G$. Theorem~3.13 of \cite{enchilada} implies that there is a functor
\[
\CP_H:\Aca^\nor(G)\to \Aa(H,(C_0(G),\rt))
\]
such that $(B,\delta)\mapsto (B\rtimes_\delta G,\hat\delta|H,j_G)$. Since the right action of $H$ is free and proper, we can apply the general theory of \cite{aHKRWproper-n}, and there is a fixed-point algebra $\Fix(B\rtimes_\delta G,\hat\delta|H,j_G)$. The discussion in \cite[\S6]{kqrproper} shows that this fixed-point algebra, which is a subalgebra of $M(B\rtimes_{\delta} G)$, coincides with the reduced crossed product by the homogeneous space $G/H$, which is by definition the closed span
\[
B\rtimes_{\delta,r} (G/H):=\clsp\{ j_B(b)j_G|(f):b\in B, f\in C_0(G/H) \}
\subset M(B\rtimes_{\delta} G),
\]
and is by \cite[Proposition~8]{man} a $C^*$-subalgebra of $M(B\rtimes_{\delta} G)$. We proved in \cite[Proposition~5.5]{aHKRWproper-n} that $(B,\delta)\mapsto B\rtimes_{\delta,r} (G/H)$ is the object map in a functor 
$\RCP_{G/H}:\Aca^\nor(G)\to \CC$, and that this functor coincides with $\Fix\circ \CP_H$.

\subsection{Identifications of multipliers}\label{sec-he-he}
Suppose that $G$ acts freely and properly on $T$, and   $(A,\alpha,\phi)$ is an object in the semi-comma category $\Aa(G,(C_0(T),\rt))$.

\begin{prop}\label{nondegeneracy}
The set $E(A_0)A_0$ is dense in $A_0$ in the $C^*$-norm of $A$. 
\end{prop}

The proof uses a lemma, which follows from a routine compactness argument.

\begin{lemma}\label{routinecpct}
Suppose that $K$ is a compact subset of a $C^*$-algebra $A$ and $\epsilon>0$. Then there exists $a\in A_0$ such that $\|a\|\leq 1$ and $\|ab-b\|<\epsilon$ for all $b\in K$.
\end{lemma}

\begin{proof}[Proof of Proposition~\ref{nondegeneracy}]
Let $b\in A_0$, and choose $f\in C_c(T)$ such that $b=fb$. Choose $g\in C_c(T)$ such that $E^{\rt}(g)f=f$ (see \cite[Lemma~2.7]{kqrproper}). Properness implies that $M:=\{s\in G:\rt_s(g)f\not=0\}$ is compact, and then we can apply Lemma~\ref{routinecpct} to $K:=\{\alpha_s^{-1}(\rt_s(g)fb):s\in M\}$. This gives $a\in A_0$ such that $\alpha_s(a)\rt_s(g)fb\sim \rt_s(g)fb$ uniformly for $s\in M$. Now \cite[Lemma~2.2]{kqrproper} gives 
\begin{align*}
E(ag)fb&=\int_M\alpha_s(ag)fb\,ds=\int_M\alpha_s(a)\rt_s(g)fb\,ds\\
&\sim\int_M\rt_s(g)fb\,ds=E^{\rt}(g)fb=fb=b.\qedhere
\end{align*}
\end{proof}

\begin{cor}\label{fixKQR}
The inclusion $\iota_G$ of $A^\alpha:=\Fix(A,\alpha,\phi)$ in $M(A)$ is nondegenerate, and extends to an isomorphism of $M(A^\alpha)$ onto 
\[\{m\in M(A):m\iota_G(A^\alpha)\subset\iota_G(A^\alpha)\ \text{and}\  \iota_G(A^\alpha)m\subset\iota_G(A^\alpha)\}\subset M(A).
\]
\end{cor}

Corollary~\ref{fixKQR} was implicitly assumed in  the proof of \cite[Proposition~2.6]{kqrproper}, and hence also indirectly in \cite{aHKRWproper-n} whenever we applied that proposition. To see the problem, recall that \cite[Proposition~2.6]{kqrproper} says that a nondegenerate homomorphism $\sigma:A\to M(B)$ in the comma category ``restricts'' to a nondegenerate homomorphism of $\Fix(A,\alpha,\phi)$ into $M(\Fix(B,\beta,\psi))$. In the proof, they show that the strictly continuous extension $\overline{\sigma}:M(A)\to M(B)$ maps $E(A_0)$ into $\{m\in M(B):m(\Fix B)\cup (\Fix B)m\subset \Fix B\}$, but then they need Corollary~\ref{fixKQR} to identify this with $M(\Fix B)$, so that the restriction $\overline{\sigma}|:\Fix A\to M(B)$ can be viewed as a homomorphism $\sigma|:\Fix A\to M(\Fix B)$. 

In \cite{kqrproper} and \cite{aHKRWproper-n}, nondegenerate homomorphisms were silently extended to multiplier algebras. In this paper, where there are often several different fixed-point algebras around, these issues can be a little slippery, and we often make the extensions explicit to avoid confusion. In particular, the inclusions $\iota_G$ and their extensions to multiplier algebras crop up a lot. When we take these issues into account, we restate Proposition~2.6 of \cite{kqrproper} as follows: 

\begin{prop}\label{newkqr}
Suppose that $\sigma:(A,\alpha,\phi)\to (B,\beta,\psi)$ is a morphism in the comma category $(C_0(T),\rt)\downarrow\Aa_{\nd}(G)$ of \cite{kqrproper}. Then there is a nondegenerate homomorphism $\sigma_G$ of $\Fix(A,\alpha,\phi)$ into $M(\Fix(B,\beta,\psi))$ such that $\overline{\iota_G}\circ\sigma_G$ is the restriction $\overline\sigma\circ\iota_G$ of $\overline{\sigma}:M(A)\to M(B)$ to the subalgebra $\Fix(A,\alpha,\phi)$ of $M(A)$.
\end{prop}

\section{Fix and Green induction}\label{sec:fixandGreen}

Let $\alpha:G\to \Aut A$ be an action of a locally compact group on a $C^*$-algebra, and let $H$ be a closed subgroup of $G$.  In \cite[\S2]{green}, Green constructed an
 \[((A\otimes C_0(G/H))\rtimes_{\alpha\otimes\lt}G)-(A\rtimes_{\alpha|} H)\] 
 imprimitivity bimodule $X(A,\alpha)$, and defined induction of representations by applying Rieffel's theory to the right-Hilbert $(A\rtimes_\alpha G)$\,--\,$(A\rtimes_{\alpha|} H)$ bimodule $X_H^G(A,\alpha)$ obtained from $X(A,\alpha)$ using the canonical map of $A\rtimes_\alpha G$ into $M((A\otimes C_0(G/H))\rtimes_{\alpha\otimes\lt}G)$ (the details of this process are also discussed in \cite[pages~5-6]{aHKRW-JOT}). Let  $I$ and $J$ be the kernels of the quotient maps of $A \rtimes_{\alpha}G$ and $A\rtimes_{\alpha|} H$ onto the corresponding reduced crossed products.  Induction-in-stages, as in \cite[Proposition~8]{green}, implies that $X_H^G(A,\alpha)\dashind$ takes regular representations to regular representations, so $I=X_H^G(A,\alpha)\dashind J$, and the corresponding quotient $X_{H,r}^G(A,\alpha)$ of $X_H^G(A,\alpha)$ is a right-Hilbert $(A\rtimes_{\alpha,r}G)$\,--\,$(A\rtimes_{\alpha|,r} H)$ bimodule.  Both $X_H^G(A,\alpha)$ and $X_{H,r}^G(A,\alpha)$ are completions of $C_c(G,A)$, viewed as a $C_c(G,A)$\,--\,$C_c(H,A)$ bimodule with the inner products and actions given in \cite[Equations B.5]{enchilada}.

It is a recurring theme in nonabelian duality that duality swaps  induction and restriction of representations \cite{echterhoff,ekr,enchilada, aHKRW-JOT,kqrold}. Our next theorem shows that this is a general phenomenon in our semi-comma category.

\begin{thm}\label{thm-extension}
  Suppose that a locally compact group $G$ acts freely and
  properly on a locally compact space $T$, and $H$ is a closed
  subgroup of $G$. Let $(A,\alpha,\phi)$ be an object in the
   semi-comma category $\Aa(G,(C_0(T),\rt))$. Then $(A,\alpha|{H},\phi)$ is an object in
  $\Aa(H,(C_0(T),\rt))$, and the following diagram commutes in $\CC$:
  \begin{equation}\label{eq-extension}
    \xymatrix{
      A\rtimes_{\alpha,r} G
      \ar[rrr]^{Z(A,\alpha,\phi)}
      \ar[d]_{X_{H,r}^G(A)}
      &&&
      \Fix(A,\alpha,\phi)
      \ar[d]^{\Fix(A,\alpha|{H},\phi)}
      \\
      A \rtimes_{\alpha|,r} H
      \ar[rrr]^{Z(A,\alpha|H,\phi)}
      &&&
      \Fix(A,\alpha|{H},\phi).
    }
  \end{equation}
\end{thm}

Our proof of this theorem is similar to that of \cite[Theorem~3.1]{kqrold}.
The next lemma  seems more elementary to us than the result of Mansfield \cite[Lemma~25]{man} used in \cite[Theorem~3.1]{kqrold}, but we lack a reference. 

\begin{lemma}\label{lem-dense} Suppose $(A,\alpha,\phi)$ is an object in $\Aa(G,(C_0(T),\rt))$. Then the subset $L:=\lsp\{s\mapsto a\alpha_s(b):a,b\in A_0\}$ of $C_c(G,A)$ is dense in Green's imprimitivity bimodule $X_H^G$ (and hence also in $X_{H,r}^G$).
\end{lemma}

\begin{proof}
When we proved in \cite[Lemma~C.1]{aHRWproper2} that $\alpha$ is saturated, we showed that $\lsp\{s\mapsto a\alpha_s(b)\Delta_G(s)^{-1/2}: a,b\in A_0\}$ is dense in $C_c(G,A)$ in the inductive limit topology, and this implies that  $L$ is dense in $C_c(G,A)$ in the inductive limit topology.

We need to show that we can approximate any element $x$ of $C_c(G,A)$ in the $X^G_H$-norm by elements of $L$. We know that there are a compact set $K\subset G$ and $x_n\in L$ such that $\supp x_n \subset K$ and $\|x_n-x\|_\infty\to 0$. For any $y\in C_c(G,A)$, we have
\begin{align*}
\langle y\,,\, y\rangle_{A\rtimes_{\alpha|} H}(h)&=\Delta_H(h)^{-1/2}\int_G\alpha_t\big(y(t^{-1})^*y(t^{-1}h)\big)\, dt\notag\\
&=\Delta_H(h)^{-1/2}\int_G y^*(t)\alpha_t(y(t^{-1}h))\Delta_G(t)\, dt\notag\\
&=\Delta_H(h)^{-1/2}\big(  (y\Delta_G^{-1})^**y\big)(h).
\end{align*}
Thus the support of $\langle x-x_n\,,\, x-x_n\rangle_{A\rtimes H}$ is contained in $K^{-1}K$, and
\begin{align}
\|x-x_n\|^2_{X^G_H}&=\|\langle x-x_n\,,\, x-x_n\rangle_{A\rtimes H}\|\notag
=\|\Delta_H^{-1/2}\big(  (x-x_n)\Delta_G^{-1})^**(x-x_n)\big)\|\notag
\\
&\leq \|\Delta_H^{-1/2}\big(  (x-x_n)\Delta_G^{-1})^**(x-x_n)\big)\|_{L^1(H,A)}\notag\\
&\leq C(H, G,K)\|x_n-x\|^2_{L^1(H,A)}\label{eq-ilt}
\end{align}
where $C(H, G,K)=\|(\Delta_H^{-1/2}\Delta_G^{-1})|K\|_\infty\mu(K^{-1})$.  Since $\|x_n-x\|_\infty\to 0$, we also have $x_n|H\to x|H$ in the inductive limit topology on $C_c(H,A)$, and hence $\|x-x_n\|_{L^1(H,A)}\to 0$.  Now \eqref{eq-ilt} implies that  $x_n\to x$ in $X_H^G$.
\end{proof}

\begin{proof}[Proof of Theorem~\ref{thm-extension}]
Fix $a,b\in A_0$. Note that $E^G(a)$ and $E^H(b)$ are both in $M(A)$.  Moreover, $E^G(a)b\in A_0$ by \cite[Lemma~2.3]{kqrproper} and $E^G(a)E^H(b)=E^H(E^G(a)b)$ by  \cite[Lemma~2.1]{kqrproper}.  
Thus  $\Fix(A,\alpha|H,\phi)$ is a right-Hilbert $\Fix(A,\alpha,\phi)$\,--\,$\Fix(A,\alpha|H,\phi)$ bimodule, and the right vertical arrow in the diagram \eqref{eq-extension} makes sense.
Since the horizontal arrows in  \eqref{eq-extension} are invertible, to see that \eqref{eq-extension}  commutes it suffices to show 
  \begin{equation}\label{eq-ext-iso}
  Z(A,\alpha,\phi)\otimes_{\Fix(A,\alpha, \phi)}Z(A,\alpha|H,\phi)^\sim\cong X_{H,r}^G(A)
  \end{equation}
  as right-Hilbert $(A\rtimes_{\alpha,r}G)$\,--\,$(A\rtimes_{\alpha|,r}H)$ bimodules.

Define $\Omega:A_0\otimes_{\Fix(A,\alpha, \phi)} A_0^\op\to C_c(G,A)\subset X_{H,r}^G$  by $\Omega(a\otimes\flat(b))(s)=a\alpha_s(b^*)$. 
We will show that $\Omega$ extends to give the isomorphism \eqref{eq-ext-iso}.
Since we know from Lemma~\ref{lem-dense} that $\Omega$ has dense range, it suffices to show that $\Omega$ preserves the $(A\rtimes_{\alpha|,r}H)$-valued inner products and is $(A\rtimes_{\alpha,r}G)$-linear.

Fix  $a,b,c,d\in A_0$. We compute the $(A\rtimes_{\alpha|,r}H)$-valued inner products: 
  \begin{align*}
  \langle\Omega(a\otimes\flat(b))&\,,\, \Omega(c\otimes\flat(d))\rangle_{A\rtimes_{\alpha|, r}H}(h)\\
  &=\Delta_G(h)^{-1/2}\int_G\alpha_t\big(\Omega(a\otimes\flat(b))(t^{-1})^*\Omega(c\otimes\flat(d))(t^{-1}h)\big)\, dt\\
   &=\Delta_G(h)^{-1/2}\int_G\alpha_t\big( (a\alpha_{t^{-1}}(b^*))^*c\alpha_{t^{-1}h}(d^*) \big)\, dt\\
    &=\Delta_G(h)^{-1/2}\left(\int_G b\alpha_t(a^*c)\, dt \right)\alpha_h(d^*)=\Delta_G(h)^{-1/2}bE^G(a^*c)\alpha_h(d^*);
  \end{align*}
  \begin{align*}
  \langle a\otimes\flat(b)\,,\, c\otimes\flat(d)\rangle_{A\rtimes_{\alpha|,r}H}(h)
  &=\big\langle \langle c\,,\, a\rangle_{\Fix(A,\alpha, \phi)}\cdot\flat(b)\,,\, \flat(d)  \big\rangle_{A\rtimes_{\alpha|,r}H}(h)\\
  &=\big\langle E^G(c^*a)\cdot\flat(b)\,,\, \flat(d) \big\rangle_{A\rtimes_{\alpha|,r}H}(h)\\
  &=\langle\flat(b\cdot E^G(c^*a)^*\,,\, \flat(d)\rangle_{A\rtimes_{\alpha|,r}H}(h)\\
  &={}_{A\rtimes_{\alpha|,r}H}\langle b\cdot E^G(c^*a)^*\,,\, d\rangle(h)\\
  &=\Delta_G(h)^{-1/2}bE^G(a^*c)\alpha_h(d^*),
  \end{align*}
  so $\Omega$ is isometric.  If $z\in C_c(G,A)\subset A\rtimes_{\alpha,r}G$ and $x\in X_{H,r}^G$, then $z\cdot x$ is given by the formula $(z\cdot x)(s)=\int_G z(t) \alpha_t(x(t^{-1}s))\Delta_G(t)^{1/2}\, dt$ (see \cite[\S2]{aHKRW-JOT}).  Thus
  \begin{align*}
  (z\cdot\Omega(a\otimes\flat(b)))(s)
   &=\int_G z(t)\alpha_t\big( \Omega(a\otimes\flat(b))(t^{-1}s) \big)\Delta_G(t)^{1/2}\, dt\\
  &= \int_G z(t)\alpha_t(a\alpha_{t^{-1}s}(b^*))\Delta_G(t)^{1/2}\, dt\\
  &= \left( \int_G z(t)\alpha_t(a)\Delta_G(t)^{1/2}\, dt\right)\alpha_s(b^*)\\
  &=(z\cdot a)\alpha_s(b^*)
  =\Omega(z\cdot a\otimes\flat(b))(s).
  \end{align*}
  Hence $\Omega$ extends to an isomorphism, as required.
\end{proof}

Theorem~\ref{thm-extension} has three direct outcomes. First we generalise Theorem~3.1 of \cite{kqrold} to non-normal subgroups, by applying Theorem~\ref{thm-extension} with $(T,G,H)=(G,K,H)$ and $(A,\alpha,\phi)=(B\rtimes_\delta G,\hat\delta,j_G)$.

\begin{cor}
  \label{prop-prob2homog}  Suppose that $H$ and $K$ are closed subgroups of a locally compact group $G$ with $H\subset K$, and $(B,\delta)$ is an object in $\Aca^\nor(G)$.
Then the following diagram commutes in~$\CC$:
  \begin{equation*}
    \xymatrix{
      (B\rtimes_\delta G) \rtimes_{\hat\delta, r}K
      \ar[rrr]^{Z(B\rtimes_\delta G,\hat\delta|_K,j_G)}
      \ar[d]_{X_{H,r}^K(B\rtimes_\delta G)}
      &&&
      B\rtimes_{\delta,r}(G/K)
      \ar[d]^{B\rtimes_{\delta,r}(G/H)}
      \\
      (B\rtimes_\delta G) \rtimes_{\hat\delta,r}H
      \ar[rrr]^{Z(B\rtimes_\delta G,\hat\delta|_H,j_G)}
      &&&
      B\rtimes_{\delta, r}(G/H).
    }
  \end{equation*}
\end{cor}

The second outcome of Theorem~\ref{thm-extension} characterises representations of $\Fix(A,\alpha, \phi)$ which extend to representations of $\Fix(A,\alpha|H,\phi)$, in the spirit of \cite[Theorem~6.1]{aHKRW-JOT}.  Indeed, applying Proposition~\ref{prop-rep} with $(T,G,H)=(G,K,H)$ gives a direct generalisation of Theorem~6.1 of \cite{aHKRW-JOT} to pairs of subgroups $H\subset K$ with either $H$  normal in $K$ or $H$ amenable. 

\begin{prop}\label{prop-rep}
 Suppose that locally compact group $G$ acts freely and
  properly on a locally compact space $T$, and $H$ is a closed
  subgroup of $G$. Assume that $H$ is either normal or amenable. Let $(A,\alpha,\phi)$ be an object in $\Aa(G,(C_0(T),\rt))$. Let $\omega$ be a representation of $\Fix(A,\alpha,\phi)$ on a
  Hilbert space $\H$, and denote by $(\rho,U)$ the covariant
  representation of $(A, G,\alpha)$ such that
  $\rho\rtimes U=Z(A,\alpha,\phi)\dashind \omega$.  Then there is a
  representation $\eta$ of $\Fix(A,\alpha|H,\phi)$ on $\H$ such that
  $\omega=\eta|\Fix (A,\alpha,\phi)$ if and only if there is a representation $\phi$
  of $C_0(G/H)$ in the commutant $\rho(A)'$ such that
  $(\phi, U)$ is a covariant representation of $(C_0(G/H), G,\lt)$.
\end{prop}

\begin{proof}   The proof proceeds as in \cite[Theorem~6.1]{aHKRW-JOT}.
  Since the horizontal arrows in \eqref{eq-extension} are bijections, $\omega$ is in the image
  of $\Res:=\Fix(A,\alpha|,\phi)\dashind$ if and only if $\rho\rtimes U$ is in the image of
  $X_H^G(A)\dashind$.  Green's imprimitivity theorem
  (the usual version if $H$ is amenable, or \cite[Theorem~6.2]{aHKRW-JOT} if $H$ is normal) says that $\rho\rtimes U$ is in the
  image of $X_H^G(A)\dashind$ if and only if there
  is a representation $\phi$ of $C_0(G/H)$ in the commutant of
  $\rho(A)$ such that $(\phi, U)$ is a covariant
  representation of $(C_0(G/H), G,\lt)$.
\end{proof}

The third outcome of Theorem~\ref{thm-extension} takes more work. 

\begin{thm}\label{thm-green}
Let $H$ and $K$ be closed subgroups of $G$ with $H\subset K$.

\textnormal{(1)} The reduced Green bimodules $X_{H,r}^K(A,\alpha)$ implement  a natural transformation between the functors $\RCP_H:(A,\alpha,\phi)\mapsto A\rtimes_{\alpha,r}H$ and $\RCP_K$ on $\Aa(G,(C_0(T),\rt))$.

\textnormal{(2)} The bimodules implementing restriction of representations implement a natural transformation between the functors $\Fix_H:(A,\alpha,\phi)\mapsto \Fix(A,\alpha|H,\phi)$ and $\Fix_K$ on $\Aa(G,(C_0(T),\rt))$.
\end{thm}
    
In fact, naturality of induction (part (1) of Theorem~\ref{thm-green}) is equivalent to naturality of restriction (part (2) of Theorem~\ref{thm-green}). To see this, suppose that $[W,u]$ is a morphism in $\Aa(G,(C_0(T),\rt))$ from $(A,\alpha,\phi)$ to $(B,\beta,\psi)$, and consider the diagram
\begin{equation}\label{astrid'scube}
\xymatrix@!0@R=40pt@C=80pt{
A\rtimes_{\alpha|,r}K
\ar[rrr]^{Z(A,\alpha|K,\phi)}
\ar[ddd]_{X_{H,r}^K(A,\alpha)}
\ar[rd]^(.6){W\rtimes_{u|,r}K}
&
&
&
\Fix(A,\alpha|K,\phi)
\ar'[d][ddd]^{\Fix(A,\alpha|H)}
\ar[rd]^(.6){\Fix(W,u|K)}
&
\\
&
B\rtimes_{\beta|,r}K
\ar[rrr]^{Z(B,\beta|K,\psi)}
\ar[ddd]_{X_{H,r}^K(B,\beta)}
&
&
&
\Fix(B,\beta|K,\psi)
\ar[ddd]^{\Fix(B,\beta|H)}
\\
&&&&\\
A\rtimes_{\alpha|,r}H
\ar'[r][rrr]^(.4){Z(A,\alpha|H,\phi)}
\ar[rd]_(0.4){W\rtimes_{u|,r}H\;}
&
&
&
\Fix(A,\alpha|H,\phi)
\ar[rd]_(.4){\Fix(W,u|H)\;}
&
\\
&
B\rtimes_{\beta|,r}H
\ar[rrr]^{Z(B,\beta|H,\psi)}
&
&
&
\Fix(B,\beta|H,\psi).
}
\end{equation}
Part (1) says that the left-hand face commutes and part (2) that the right-hand face commutes.
The top and bottom faces commute by \cite[Theorem~3.5]{aHKRWproper-n}, and the front and back faces commute by Theorem~\ref{thm-extension}. So, since the left-right arrows are all isomorphisms in our category (that is, are implemented by imprimitivity bimodules), the left-hand face commutes if and only if the right-hand face commutes, and part~(1) is equivalent to part~(2).

So we can prove the theorem by proving either (1) or (2). On the face of it (sorry), it would seem to be easier to prove that a square involving restriction maps commutes. However, when we apply this to nonabelian duality, the corners in the right-hand face will be crossed products by homogeneous spaces $G/H$ and $G/K$, and since representations of such crossed products are not given by covariant pairs, ``restriction'' doesn't obviously have an intuitive meaning. 
So we complete the proof of Theorem\ref{thm-green} with the following lemma.

\begin{lemma}\label{redGreennat}
The left-hand face of \eqref{astrid'scube} commutes.
\end{lemma}

The analogue of the left-hand face for full crossed products commutes by \cite[Theorem~4.1]{taco}, and we will deduce from this that the reduced version commutes. To do this, we need two lemmas.

\begin{lemma}\label{quotientbimod}
Suppose that $\theta:{}_AX_B\to {}_AY_B$ is an isomorphism of right-Hilbert bimodules, and $J$ is an ideal in $B$. Then $X\dashind J=Y\dashind J$. If $I$ is an ideal in $A$ which is contained in $X\dashind J$, then there is a right-Hilbert $A/I$\,--\,$B/J$ bimodule isomorphism $\theta^J$ of $X^J:=X/XJ$ onto $Y^J$ such that $\theta^J(q(x))=q(\theta(x))$ for $x\in X$.
\end{lemma}

\begin{proof}  (We use the same letter $q$ for several different quotient maps.)
Notice that $\theta$ maps $XJ$ onto $YJ$, and hence induces a linear isomorphism of $X^J$ onto $Y^J$ such that $\theta^J(q(x))=q(\theta(x))$, and it follows from \cite[Proposition~2.5]{tfb} that $\theta^J$ is an isomorphism of right-Hilbert $B/J$--modules. For $a\in A$, the definition of the induced ideal (in \cite[Definition~1.7]{enchilada}) gives
\begin{align*}a\in X\dashind J
&\Longleftrightarrow a\cdot x=0\text{ for all $x\in X$}\\
&\Longleftrightarrow a\cdot x\in XJ\text{ for all $x\in X$ (by \cite[Lemma~3.23]{tfb})}\\
&\Longleftrightarrow a\cdot \theta(x)\in\theta(XJ)=YJ\text{ for all $x\in X$}\\
&\Longleftrightarrow a\in Y\dashind J.
\end{align*}
Thus the left actions of $A$ on $X^J$ and $Y^J$ pass to left actions of $A/(Y\dashind J)$ and $A/I$ such that $q(a)\cdot q(x)=q(a\cdot x)$, and $\theta^J$ is also an $(A/I)$-module homomorphism.
\end{proof}

\begin{lemma}\label{tensorquot}
Suppose that we have right-Hilbert bimodules $_AV_B$ and $_BY_C$, and that we have ideals $I$ in $A$, $J$ in $B$ and $L$ in $C$ satisfying $J\subset Y\dashind L$ and $I\subset V\dashind J$. Then the map $v\otimes y\mapsto q(v)\otimes q(y)$ induces an isomorphism
\[
(V\otimes_B Y)/((V\otimes_B Y)L)\cong(V/VJ)\otimes_{B/J}(Y/YL)
\]
of right-Hilbert $(A/I)$\,--\,$(C/L)$ bimodules.
\end{lemma}

\begin{proof}
A calculation shows that $q\big(\langle u\otimes x,v\otimes y\rangle_C\big)=\langle q(u)\otimes q(x),q(v)\otimes q(y)\rangle_C$, so the map is well-defined and preserves the inner product. Another calculation shows that the map preserves the left action, and  since it has dense range this suffices.
\end{proof}

\begin{proof}[Proof of Lemma~\ref{redGreennat}]
The commutativity of diagram (4.1) in \cite{taco} says that there is an isomorphism
\begin{equation}\label{fromtaco}
X_{H}^K(A,\alpha)\otimes_{A\rtimes_{\alpha|}H}(W\rtimes_{u|}H)\cong
(W\rtimes_{u|} K)\otimes_{B\rtimes_{\beta|}K}X_{H}^K(B,\beta)
\end{equation}
of right-Hilbert $(A\rtimes_{\alpha|}K)$\,--\,$(B\rtimes_{\beta|} H)$ bimodules; we write $\LHS$ and $\RHS$ for the left and right sides of \eqref{fromtaco}. Let $I\subset A\rtimes_{\alpha|} K$, $J\subset A\rtimes_{\alpha|} H$ and $L\subset B\rtimes_{\beta|} H$ be the kernels of the quotient maps onto the reduced crossed products. Then applying Lemma~\ref{quotientbimod} to \eqref{fromtaco} gives an isomorphism
\begin{equation}\label{quotiso}
\LHS/{(\LHS) L}\cong\RHS/{(\RHS) L}.
\end{equation}
Since $I=X_H^K\dashind J$ by induction-in-stages, and $J=(W\rtimes_{u|} H)\dashind L$ by \cite[Corollary on page 300]{combes}, Lemma~\ref{tensorquot} implies that
\[
\LHS/{(\LHS) L}\cong X_{H,r}^K(A,\alpha)\otimes_{A\rtimes_{\alpha,r}H}(W\rtimes_{u|,r}H).
\]
Similarly, we have
\[
\RHS/{(\RHS) L}\cong
(W\rtimes_{u|,r} K)\otimes_{B\rtimes_{\beta|,r}K}X_{H,r}^K(B,\beta),
\]
and the isomorphism in \eqref{quotiso} says that left face of \eqref{astrid'scube} commutes.
\end{proof}

\section{Fixing-in-stages}\label{sec:abstractfis}

Throughout this section  $G$ is a locally compact group acting freely and properly on a locally compact space $T$, and  $N$ is a closed normal subgroup of $G$. Restricting the actions to $N$ gives a functor from $\Aa(G,(C_0(T),\rt))$ to $\Aa(N,(C_0(T),\rt|))$. Since the action of $N$ on $T$ is also free and proper, we can then apply the functor $\Fix$, and the composition is a functor
$\Fix_N:\Aa(G,(C_0(T),\rt))\to \CC$ which assigns to each object $(A,\alpha,\phi)$ a fixed-point algebra $A^{\alpha|N}:=\Fix(A,\alpha|N,\phi)$. We will build an equivariant version $\Fix_N^{G/N}$ of $\Fix_N$ with values in $\Aa(G/N,(C_0(T/N),\rt|))$, so that there is an iterated fixed-point algebra $\Fix_{G/N}(A^{\alpha|N})$, and prove that $\Fix_{G/N}\circ\Fix_N^{G/N}$ is naturally  isomorphic to $\Fix_G$ (see Theorem~\ref{thm-newfix} below).  

Viewing functions in $C_0(T/N)$ as bounded functions on $T$ gives an embedding of $C_0(T/N)$ in $C_b(T)=M(C_0(T))$, which identifies $C_0(T/N)$ with the generalised fixed-point algebra $C_0(T)^{\rt}=\Fix (C_0(T),\rt|N,\id)$ (see, for example, \cite[Proposition~3.1]{marrae}). Now applying Proposition~\ref{newkqr} to $\phi:C_0(T)\to M(A)$ gives a nondegenerate homomorphism $\phi_N:C_0(T/N)\to M(A^{\alpha|N})$ such that $\overline{\iota_N}\circ \phi_N$ is the restriction to $C_0(T/N)$ of the strictly continuous extension $\overline{\phi}:M(C_0(T))\to M(A)$.

\begin{prop}\label{prop-actG/N}
Let $[X,u]$ be a morphism from $(A,\alpha,\phi)$ to $(B,\beta,\psi)$ in the  semi-comma category $\Aa(G,(C_0(T),\rt))$. There are actions $\alpha^{G/N}$, $\beta^{G/N}$ and $u^{G/N}$ of $G/N$ such that $(\Fix(X,u|N),u^{G/N})$ defines a morphism from $(A^{\alpha|N},\alpha^{G/N},\phi_N)$ to $(B^{\beta|N},\beta^{G/N},\psi_N)$ in the category $\Aa(G/N,(C_{0}(T/N),\rt))$, and the assignments
\begin{equation}\label{deffixgn}
    (A,\alpha,\phi)\mapsto (A^{\alpha|N},\alpha^{G/N},\phi_N)\quad\text{and}
    \quad [X,u]\mapsto [\Fix(X,u|N),u^{G/N}]
  \end{equation}
form a functor $\Fix^{G/N}_{N}$ from $\Aa(G,(C_{0}(T),\rt))$ to $\Aa(G/N,(C_{0}(T/N),\rt^{G/N}))$.
\end{prop}

To prove Proposition~\ref{prop-actG/N}, we have to define the actions, and then chase through the proof that $\Fix$ is a functor  \cite[\S4]{aHKRWproper-n} checking that the constructions are equivariant. We do this in a series of lemmas. 

We claim that each $\alpha_t$ maps $A^{\alpha|N}$ into $A^{\alpha|N}$, and then define $\alpha^{G/N}_{tN}:=\alpha_t|{A^{\alpha|N}}$.  Since $N$ is normal in $G$,
\begin{equation}\label{Delta}
\int_N \xi(n)\, dn
=\Delta_{G,N}(t)\int_N \xi(t^{-1}nt)\,dn:=\Delta_G(t)\Delta_{G/N}(tN)^{-1}\int_N \xi(t^{-1}nt)\, dn
\end{equation}
for $\xi\in C_c(N)$.   Fix $f\in C_c(T)$,  $t\in G$ and $a\in A_0$. Then, writing $E^N$ for the expectation associated with $(A,\alpha|N,\phi)$ and applying \cite[Lemma~2.2]{kqrproper}, we have
\begin{align}\label{alphaokonG/N}
f\alpha_t(E^N(a))&=\alpha_t(\alpha_t^{-1}(f)E^N(a))
=\alpha_t\Big(\int_N f\alpha_{n}(a)\, dn\Big)
=\int_N f\alpha_{(tnt^{-1})t}(a)\, dn\notag\\
&=\Delta_{G,N}(t)\int_N f\alpha_{nt}(a)\, dn=f\Delta_{G,N}(t)E^N(\alpha_t(a)).
\end{align}
Since $\phi$ is nondegenerate, this implies that $\alpha_t(E^N(a))=\Delta_{G,N}(t)E^N(\alpha_t(a))$, which belongs to $E^N(A_0)\subset A^{\alpha|N}$. The continuity of $\alpha_t$ now implies that $\alpha_t(A^{\alpha|N})\subset A^{\alpha|N}$, and we can define $\alpha^{G/N}:G/N\to\Aut(A^{\alpha|N})$ by 
$\alpha_{tN}=\alpha_t|{A^{\alpha|N}}$. To see that $\alpha^{G/N}$ is strongly continuous, let $a\in A_0$, and choose $f,g\in C_c(T)$ such that $\alpha_t(a)=f\alpha_t(a)g$ for all $t$ in a compact neighbourhood of $e$. Then the norm continuity of $a\mapsto E^N(fag)$ implies that $t\mapsto E^N(f\alpha_t(a)g)$ is continuous at $e$, and  \eqref{alphaokonG/N} implies that $t\mapsto \alpha_{tN}^{G/N}(E^N(a))$ is continuous at $e$.

Since $\phi\circ \rt_t=\alpha_t\circ\phi$ and the definition of $\alpha^{G/N}$ implies that $\iota_N\circ\alpha^{G/N}_{tN}=\alpha_t\circ\iota_N$, $\phi_N=(\overline{\iota_N})^{-1}\circ\overline\phi$ is $\rt$\,--\,$\alpha^{G/N}$ equivariant, and since $\rt^{G/N}_{tN}=\rt_t|C_0(T/N)$, the triple $(A^{\alpha|N},\alpha^{G/N},\phi_N)$ is an object in the semi-comma category. Now we have to deal with morphisms.  Since the morphisms in the semi-comma category are the same as those in $\Aa(G/N)$, we don't need to worry about the homomorphisms $\phi_N$ any more. So we take $(X,u)$ as in Proposition~\ref{prop-actG/N}, and recall the construction of $\Fix X$ from \cite[\S3]{aHKRWproper-n}, as it applies to $u|N$. Let $(\K(X),\mu,\phi_{\K})$ be the object in $\Aa(G,(C_{0}(T),\rt))$ and $\kappa:A\to M(\K(X))$ the
nondegenerate homomorphism provided by
the canonical decomposition of $[X,u]$ in
\cite[Corollary~2.3]{aHKRWproper-n}.  The actions $\mu$, $u$ and $\beta$ give an action $L(u)$ of $G$ on the linking algebra $L(X)$, and 
$(L(X),L(u),\phi_{\K}\oplus\psi)$ is an object in $\Aa(G,(C_{0}(T),\rt))$. Proposition~3.1 of \cite{aHKRWproper-n} implies that the top right-hand corner $X^{u|N}$ in $L(X)^{L(u)|N}$ is a $\K^{\mu|N}$\,--\,$B^{\beta|N}$ imprimitivity bimodule, and Proposition~\ref{newkqr} gives a
nondegenerate homomorphism $\kappa_N:A^{\alpha|N}\to M(\K^{\mu|N})$, so $X^{u|N}$ becomes a right-Hilbert $A^{\alpha|N}$\,--\,$B^{\beta|N}$ bimodule $\Fix(X,u|N)$.  

The construction described two paragraphs above applies to the action $L(u)$, giving a strongly continuous action $L(u)^{G/N}:G/N\to \Aut L(X)^{L(u)|N}$ which restricts to the actions $\beta^{G/N}$ and $\mu^{G/N}$ on the diagonal corners, and hence restricts to a compatible action $u^{G/N}$ on the corner $X^{u|N}$. The operations on the \ib{} $X^{u|N}$ come from the matrix operations in $L(X)^{L(u)|N}$, so $(X^{u|N},u^{G/N})$ is\label{defmorphism} a $(\K(X)^{\mu|N},\alpha^{G/N})$\,--\,$(B^{\beta|N},\beta^{G/N})$ \ib. The homomorphism $\kappa|N$ defining the left action is $\alpha^{G/N}$--$\mu^{G/N}$ equivariant, and hence $(X^{u|N},u^{G/N})$ is a right-Hilbert $(A^{\alpha|N},\alpha^{G/N})$\,--\,$(B^{\beta|N},\beta^{G/N})$ bimodule. 

Next we show that the map $[X,u]\mapsto [X^{u|N},u^{G/N}]$ is well-defined.

\begin{lemma}
  \label{lem-fixgnn}
  Suppose that $(A,\alpha,\phi)$ and $(B,\beta,\psi)$ are objects in the semi-comma category
  $\Aa\bigl(G,(C_{0}(T),\rt))\bigr)$. 
  If $(X,u)$ and $(Y,v)$ are isomorphic as right-Hilbert
  $(A,\alpha)$\,--\,$(B,\beta)$ bimodules, then
  $(\Fix(X,u|N),u^{G/N})$ and
  $(\Fix(Y,v|N),v^{G/N})$ are isomorphic as right-Hilbert
  $(A^{\alpha|N},\alpha^{G/N})$\,--\,$(B^{\beta|N},\beta^{G/N})$
  bimodules.
\end{lemma}

\begin{proof}
Suppose that $\Psi:(X,u)\to (Y,v)$ is an isomorphism, and $\rho:\K(X)\to \K(Y)$ is the isomorphism such that $\rho(\Theta_{x,w})=\Theta_{\Psi(x),\Psi(w)}$. It  is shown in the proof of \cite[Lemma~3.2]{aHKRWproper-n} that $\rho$, $\Psi$ and $\id_B$ give an isomorphism $\Psi_L:L(X)\to L(Y)$, and that the top-right corner of $\overline{\Psi_L}$ carries $\Fix(X,u)$ isomorphically onto $\Fix(Y,v)$. Since $\Psi_L$ is $L(u)$--$L(v)$ equivariant, and the actions $u^{G/N}$ and $v^{G/N}$ are the restrictions of $u$ and $v$ to the fixed-point modules, this isomorphism of $\Fix(X,u)$ onto $\Fix(Y,v)$ is equivariant.
\end{proof}

We have now proved that the assignments \eqref{deffixgn} are well-defined, and it remains to prove that they define a functor $\Fix^{G/N}_N$. The argument in the second paragraph of \cite[\S4.3]{aHKRWproper-n} shows that fixing the identity at an object $(A,\alpha,\phi)$ gives the identity 
$_{A^{\alpha|N}}(A^{\alpha|N})_{A^{\alpha|N}}$ at the object $A^{\alpha|N}$ in $\CC$, and adding the action $\alpha^{G/N}$ throughout shows that $\Fix^{G/N}_N$ takes $(A,\alpha,\phi)$ to the identity morphism at $\Fix_N^{G/N}(A,\alpha,\phi)$. Proving that $\Fix$ preserves composition was the hard bit of \cite[Theorem~3.3]{aHKRWproper-n}, and we have to check that all the isomorphisms constructed in its proof can be made equivariant. The next two lemmas are equivariant versions of \cite[Proposition~4.1]{aHKRWproper-n} and \cite[Theorem~4.5]{aHKRWproper-n}.

\begin{lemma}
  \label{lem-prop4.1}
Suppose that $(K,\mu,\phi_K)$,
  $(B,\beta,\psi)$ and $(C,\gamma, \zeta)$ are objects in $\Aa(G, (C_0(T),\rt))$, that
  ${}_{(K,\mu)}(X,u)_{(B,\beta)}$ is an imprimitivity bimodule, and
  that $\sigma:B\to M(C)$ is a nondegenerate homomorphism which is $\beta$\,--\,$\gamma$ equivariant and satsifies $\sigma\circ\psi=\zeta$.  Then
  \begin{equation*}
    \bigl((\Fix(X,{u}|N)\tensor_{B^{\beta|N}}C^{\gamma|N},
    u^{G/N}\tensor \gamma^{G/N}\bigr) \cong
    \bigl(\Fix(X\tensor_{B}C, {(u\tensor\gamma)}|N) ,
    (u\tensor\gamma)^{G/N}\bigr) 
  \end{equation*}
  as right-Hilbert $(K^\mu,\mu^{G/N})$\,--\,$(
  C^{\gamma},\gamma^{G/N})$ bimodules.
\end{lemma}

\begin{proof}
 In  \cite[Proposition~4.1]{aHKRWproper-n} we found
  a nondegenerate homomorphism
  \begin{equation*}
    \Phi_{L}:\bigl(L(X),L(u)\bigr)\to
    \bigl(M(L(X\tensor_{B}C)),L(u\tensor\gamma)\bigr)
  \end{equation*}
  which is a morphism in the comma category
  $\bigl(C_{0}(T),\rt\bigr)\downarrow \Aa_{\nd}(G)$, and hence also in
  $\bigl(C_{0}(T),\rt\bigr)\downarrow \Aa_{\nd}(N)$, and is
  $L(u)$\,--\,$L(u\otimes\gamma)$ equivariant. 
  Proposition~\ref{newkqr} gives a nondegenerate homomorphism $(\Phi_L)_N$ of
  $L(X)^{L(u)|N}=L(X^{u|N})$ into $M(L(X\tensor_{B}C)^{L(u\tensor\gamma)|N})=M(L((X\otimes_BC)^{(u\otimes\gamma)|N}))$, and applying Lemma~4.3 of \cite{aHKRWproper-n} gave the isomorphism $\Omega$ of \cite[Proposition~4.1]{aHKRWproper-n}: if $\Psi$ denotes the top right-hand corner of $(\Phi_L)_N$, then  $\Omega$ is given by $\Omega(m\otimes d)=\Psi(m)\cdot d$, where the action is that of $C^{\gamma|N}$ on $(X\otimes_BC)^{(u\otimes\gamma)|N}$, which is by definition induced by matrix multiplication in  $
  M(L((X\otimes_BC)^{(u\otimes\gamma)|N})$. The homomorphism $(\Phi_L)_N$ is $L(u)^{G/N}$--$L(u\otimes\gamma)^{G/N}$ equivariant, and then the calculation
  \begin{align*}
    &\begin{pmatrix}
      0& \Psi(u_{tN}^{G/N}(m)) \cdot
      \gamma_{tH}^{G/N}(d) \\ 0&0
    \end{pmatrix}=(\Phi_L)_N\Bigl(
    \begin{pmatrix}
      0&u_{tN}^{G/N}(m)\\0&0
    \end{pmatrix}\Bigr)
    \begin{pmatrix}
      0&0\\0&\gamma_{tN}^{G/N}(d)
    \end{pmatrix}
    \\
    &\hskip1truein= L(u\tensor\gamma)_{tN}^{G/N}\Bigl( \ (\Phi_L)_N\Bigl(\begin{pmatrix}
      0&m\\0&0
    \end{pmatrix}
    \Bigr)\Bigr)L(u\tensor\gamma)_{tN}^{G/N} \Bigl(
    \begin{pmatrix}
      0&0\\0&d
    \end{pmatrix}
    \Bigr) \\
    &\hskip1truein=L(u\tensor\gamma)_{tN}^{G/N}\Bigl(
    \begin{pmatrix}
      0&\Omega(m\tensor d)\\0&0
    \end{pmatrix}
    \Bigr)
  \end{align*}
shows that
  \begin{equation*}
    \Omega\bigl(u_{tN}^{G/N}(m)\tensor\gamma_{tN}^{G/N}(d)\bigr) =
    \Psi\bigl(u_{tH}^{G/N}(m)\bigr) \cdot \gamma_{tN}^{G/N}(d) =
    (u\tensor\gamma)_{tH}^{G/N}\bigl(\Omega(m\tensor d)\bigr).\qedhere
  \end{equation*}
\end{proof}

\vfill\eject

\begin{lemma}
  \label{lem-thm4.5}
  Suppose that $(A,\alpha,\phi)$,
  $(B,\beta,\psi)$ and $(C,\gamma,\zeta)$ are objects in
  $\Aa(G,(C_{0}(T),\rt))$,  and that
  $_{(A,\alpha)}(X,u)_{(B,\beta)}$ and
  $_{(B,\beta)}(Y,v)_{(C,\gamma)}$ are \ib s.  Then
  \[
  \bigl(X^{u|N}\tensor_{B^{\beta|N}}Y^{v|N},u^{G/N}\tensor
  v^{G/N}\bigr)\cong \bigl((X\tensor_{B}Y)^{(u\tensor
    v)|N}, (u\tensor v)^{G/N}\bigr)
    \] as
  $(A^{\alpha},\alpha^{G/N})$\,--\,$ (C^{\gamma},\gamma^{G/N})$ \ib s.
\end{lemma}

\begin{proof}
  In the proof of \cite[Theorem~4.5]{aHKRWproper-n}, we identify
  $\mathcal{K}((X\tensor_{B}Y)\oplus Y \oplus C)$ with
    \begin{equation*}
    F=
    \begin{pmatrix}
      A&X&X\tensor_{B} Y\\ * & B & Y \\ * & * & C
    \end{pmatrix},
  \end{equation*}
  and show that the actions $\alpha$, $\beta$ $\gamma$, $u$, $v$ and $u\otimes v$ combine to give an action $\eta$ of $G$ on $F$ such that $(F,\eta,\phi\oplus\psi\oplus\zeta)$ is an object in $\Aa(G,(C_{0}(T),\rt))$.  Then
  \begin{equation*}
    \eta^{G/N}=
    \begin{pmatrix}
      \alpha^{G/N}&u^{G/N}&(u\tensor v)^{G/N} \\ * & \beta^{G/N} &
      v^{G/N} \\ * & * & \gamma^{G/N}
    \end{pmatrix}
  \end{equation*}
  acts on
  \begin{equation*}
    F^{\eta|N} =
    \begin{pmatrix}
      A^{\alpha|N}& X^{u|N}&(X\tensor_{B}Y)^{(u\tensor v)|N} \\ * &
      B^{\beta|N} & Y^{v|N} \\ * & * & C^{\gamma|N}
    \end{pmatrix}.
  \end{equation*}
  The isomorphism of $X^{u|N}\tensor_{B^{\beta|N}} Y^{v|N}$ onto
  $(X\tensor_{B}Y)^{(u\tensor v)|N}$ is given by $x\tensor y\mapsto
  xy$, where $xy$ is by definiton the top right entry in the product in $F^{\eta|N}$ of
  \begin{equation}\label{eq:2}
    \begin{pmatrix}
      0&x&0\\0&0&0\\0&0&0
    \end{pmatrix}\quad\text{ and }\quad
    \begin{pmatrix}
      0&0&0\\0&0&y\\0&0&0
    \end{pmatrix}.
  \end{equation}
Thus the isomorphism takes $u_{tN}^{G/N}(x)\tensor v_{tN}^{G/N}(y)$ to the top-right entry of the image of the product of \eqref{eq:2} under $\eta_{tN}^{G/N}$, which is $(u\tensor v)_{tN}^{G/N}(xy)$, as required.
\end{proof}

\begin{proof}[End of the proof of Proposition~\ref{prop-actG/N}]
Suppose $(X,u)$ is a Hilbert
$(A,\alpha)$\,--\,$(B,\beta)$ bimodule  and $(Y,v)$ is a Hilbert
$(B,\beta)$\,--\,$(C,\gamma)$ bimodule.  
Now we need to  follow the proof of \cite[Theorem~3.3]{aHKRWproper-n} to see how the isomorphism there is defined.  We start with the tensor product
    \begin{equation*}
      \bigl(\Fix(X,u|N),u^{G/N}\bigr) \tensor_{B^{\beta|N}}
      \bigl(\Fix(Y,v|N), v^{G/N}\bigr),
    \end{equation*}
and use \cite[Corollary~2.3]{aHKRWproper-n} to factor this
equivariantly as
\begin{equation*}
  (\K(X^{u|N}),\mu^{G/N}) \tensor_{\K(X^{u|N})}
  (X^{u|N},u^{G/N}) \tensor_{B^{\beta|N}}
  (\K(Y^{v|N}),\rho^{G/N}) \tensor_{\K(Y^{v|N})}
  (Y^{v|N},v^{G/N}) .
\end{equation*}
Lemma~\ref{lem-prop4.1} implies that the above is equivariantly
isomorphic to
\begin{equation*}
(\K(X^{u|N}),\mu^{G/N}) \tensor_{\K(X^{u|N})}
  \bigl((X\tensor_{B}\K(Y))^{(u\tensor\rho)|N},(u\tensor
  \rho)^{G/N}\bigr)  \tensor_{\K(Y^{v|N})}
  (Y^{v|N},v^{G/N}) .
\end{equation*}
By Lemma~\ref{lem-thm4.5}, this is equivariantly isomorphic to
\begin{equation*}
  \bigl(\K(X^{u|N}),\mu^{G/N}\bigr) \tensor_{\K(X^{u|N})}
  \bigl((X\tensor_{B}\K(Y)\tensor _{\K(Y)} Y)^{(u\tensor\rho\tensor
    v)|N}, (u\tensor\rho\tensor v)^{G/N}\bigr) .
\end{equation*}
In the proof of \cite[Theorem~3.3]{aHKRWproper-n}, we provide a right-Hilbert module isomorphism
\begin{equation*}
  \theta|:(X\tensor_{B}\K(Y)\tensor_{\K(Y)}Y)^{(u\tensor\rho\tensor
    v)|N}\to (X\tensor_{B}Y)^{(u\tensor v)|N},
\end{equation*}
and  it is enough to prove that $\theta|$ is
$(u\tensor\rho\tensor v)^{G/N}$\,--\,$(u\tensor v)^{G/N}$ equivariant.  

Recall that $\theta|$ is defined as follows.  We first let
$\theta:X\tensor_{B}\K(Y)\tensor_{\K(Y)}Y\to X\tensor_{B}Y$ be
the right $(C,\gamma)$-homomorphism given by $(x\tensor \Theta \tensor
y) \mapsto x\tensor \Theta(y)$.  Then 
\begin{equation*}
  L(\theta):=
  \begin{pmatrix}
    \Ad \theta&\theta\\ * & \id
  \end{pmatrix}
\end{equation*}
is an isomorphism in
$\Aa(G,(C_{0}(T),\rt))$ which maps $\bigl(L(X\tensor_{B}\K(Y)\tensor_{\K(Y)}Y),L(u\tensor
\rho\tensor v)\bigr)$ onto $\bigl(L(X\tensor_{B}Y),L(u\tensor
v)\bigr)$.  Then Proposition~\ref{newkqr} gives an isomorphism
$L(\theta)|$ on fixed-point algebras, and $\theta|$ is the top right-hand corner of $L(\theta)|$. The restriction $L(\theta)|$ is $L(u\tensor\rho\tensor v)^{G/N}$\,--\,$L(u\tensor
v)^{G/N}$ equivariant, and this equivariance implies that $\theta|$ is
$(u\tensor\rho\tensor v)^{G/N}$\,--\,$(u\tensor v)^{G/N}$ equivariant, as required.
\end{proof}

\begin{thm}\label{thm-newfix}
Suppose that a locally compact group $G$ acts freely and properly on a locally compact space $T$, and $N$ is a closed normal subgroup of $G$. For every object $(A,\alpha,\phi)$ in the  semi-comma category $\Aa(G,(C_0(T),\rt))$, the injection $\overline{\iota_N}$ of $M(A^{\alpha|N})$ in $M(A)$ described in Corollary~\ref{fixKQR} restricts to an isomorphism $\iota_{G,N}=\iota_{G,N}(A,\alpha,\phi)$ of  $\Fix(A^{\alpha|N},\alpha^{G/N},\phi_N)$ onto $A^\alpha:=\Fix(A,G,\phi)$, and  these isomorphisms give a natural isomorphism between $\Fix_{G/N}\circ \Fix_N^{G/N}$ and $\Fix$.
\end{thm}

\begin{proof}
Corollary~\ref{fixKQR} implies that $\overline{\iota_N}$ is an injection, so for the first part it suffices to check that $\overline{\iota_N}$ maps $\Fix_{G/N}(A^{\alpha|N},\alpha^{G/N},\phi_N)$ onto $\Fix(A,\alpha,\phi)$. Let $a\in A_0$. We start by showing that  $E^N(a)$ belongs to the domain $(A^{\alpha|N})_0$ of $E^{G/N}$. Choose $f,g\in C_c(T)$ such that $a=\phi(f)a\phi(g)$, and  $h,k\in C_c(T/N)$ such that $h\equiv 1$ on $(\supp f)/N$ and $k\equiv 1$ on $(\supp g)/N$. Then $\overline{\phi}(h)$ and  $\overline{\phi}(k)$ are $\alpha|N$-invariant, so
\begin{align*}
E^N(\phi(f)a\phi(g))&=\iota_N\big(E^N(\phi(hf)a\phi(gk))\big)=\iota_N\big(\overline{\phi}(h)E^N(\phi(f)a\phi(g))\overline{\phi}(k)\big)\\
&=\phi_N(h)\iota_N\big(E^N(\phi(f)a\phi(g))\big)\phi_N(k)\\
&=\phi_N(h)E^N(\phi(f)a\phi(g))\phi_N(k)
\end{align*} 
belongs to $(A^{\alpha|N})_0=\phi_N(C_c(T/N))A^{\alpha|N}\phi_N(C_c(T/N))$, and $E^{G/N}(E^N(a))$ makes sense.

Next we compute $E^G(a)$. Let $f\in C_c(T)$, and again choose $h\in C_c(T/N)$ such that $h\equiv 1$ on $(\supp f)/N$. Several applications of  \cite[Lemma 2.2]{kqrproper} show that $\phi(f)E^G(a)$ is given by the norm-convergent $A$-valued integrals
\begin{align}
\phi(f)E^G(a)
&=\int_G \phi(f)\alpha_t(a)\, dt
=\int_{G/N}\int_N \phi(f)\alpha_{tn}(a)\, dn\, d(tN)\notag\\
&=\int_{G/N}\alpha_t\left( \int_N\alpha_t^{-1}(\phi(f))\alpha_n(a)\, dn \right)\, d(tN)\notag\\
&=\int_{G/N}\alpha_t\left(\phi(\rt_t^{-1}(f)) E^N(a)\right)\, d(tN)\notag=\int_{G/N} \phi(f)\alpha_t(E^N(a))\, d(tN)\notag\\
&=\int_{G/N} \phi(f)\overline{\phi}(h)\alpha^{G/N}_{tN}(E^N(a))\, d(tN)\notag\\
&=\phi(f)\int_{G/N}\overline{\iota_N}(\phi_N(h))\alpha^{G/N}_{tN}(E^N(a))\, d(tN).\label{normcvgtok}
\end{align}
Since $E^N(a)$ belongs to $A^{\alpha|N}$, \eqref{normcvgtok} is the product in $M(A)$ of $\phi(f)$ with the image under $\iota_N$ of the norm-convergent $A^{\alpha|N}$-valued integral 
\[
\int_{G/N}\phi_N(h)\alpha^{G/N}_{tN}(E^N(a))\, d(tN)=\phi_N(h)E^{G/N}(E^N(a)).
\] 
Thus
\begin{align*}
\phi(f)E^G(a)&=\phi(f)\iota_N\big(\phi_N(h)E^{G/N}(E^N(a))\big)\\
&=\phi(f)\overline{\phi}(h)\overline{\iota_N}(E^{G/N}(E^N(a)))\\
&=\phi(f)\overline{\iota_N}(E^{G/N}(E^N(a))).
\end{align*}
Since $\phi$ is nondegenerate, this implies that
\begin{equation}\label{eqnforfixfix}
E^G(a)=\overline{\iota_N}(E^{G/N}(E^N(a)))=\iota_{G,N}(E^{G/N}(E^N(a)))\ \text{ for all $a\in A_0$.}
\end{equation}

Since
$\Fix(A^{\alpha|N},\alpha^{G/N},\phi_N)$
is by definition
\begin{equation*}
\clsp\{E^{G/N}(\phi_N(h)b\phi_N(k)):b\in A^{\alpha|N}, h,k\in C_c(T/N)\},
\end{equation*}
and since $E^N(A_0)$ is dense in $A^{\alpha|N}$, the norm continuity of $b\mapsto E^{G/N}(\phi_N(h)b\phi_N(k))$ and the argument in the first paragraph of the proof imply that 
\[
\Fix(A^{\alpha|N},\alpha^{G/N},\phi_N)=\clsp\{E^{G/N}(E^N(a))):a\in A_0\}.
\]
Thus we can deduce from Equation~\ref{eqnforfixfix}, first, that $\overline{\iota_N}(\Fix(A^{\alpha|N}))$ is contained in $\Fix(A,\alpha,\phi)$, and, second, that it is dense in $\Fix(A,\alpha,\phi)$. Since $\overline{\iota_N}$ is a homomorphism between $C^*$-algebras, it has closed range, and hence it is onto, as required.

To establish naturality, we consider a right-Hilbert $(A,\alpha)$\,--\,$(B,\beta)$ bimodule $(X,u)$ and its canonical factorisation, as in  \cite[Corollary~2.3]{aHKRWproper-n}:
\begin{equation*}  \xymatrix{
    (A,\alpha,\phi_A)    \ar[r]^-{\kappa}
    &(\K,\mu,\phi_\K)
    \ar[r]^-{X}
    &(B,\beta,\phi_B).
  }
\end{equation*}
Next, we apply Proposition~\ref{newkqr} iteratively to get maps $\kappa_N:A^{\alpha|N}\to M(\K^{\mu|N})$ and $(\kappa_N)_{G/N}:\Fix(A^{\alpha|N})\to M(\Fix(\K^{\mu|N}))$.
We need to prove that the following diagram commutes:
\begin{equation*}
  \xymatrix{
    \Fix(A^{\alpha|N},\alpha^{G/N})
    \ar[rrrr]^-{\iota_{G,N}(A)}
    \ar[dd]^-{\Fix(\Fix(X,u|))}
    \ar[rd]^-{(\kappa_N)_{G/N}}
    &&&&
    \Fix(A,\alpha)
    \ar[ld]_-{\kappa_G}
    \ar[dd]^-{\Fix(X,u)}
    \\
    & \Fix(\K^{\mu|N},\mu^{G/N})
    \ar[ld]^-{\ \ \Fix(X^{u|N},u^{G/N})}
    \ar[rr]^-{\iota_{G,N}(\K)}
    &&\Fix(\K,\mu)
    \ar[rd]^-{({}_K X_B)^u}
    \\
    \Fix(B^{\beta|N},\beta^{G/N})
    \ar[rrrr]^-{\iota_{G,N}(B)}
    &&&&
    \Fix(B,\beta).
  }
\end{equation*} 
The left and right triangles commute because the functor $\Fix$ is defined on morphisms by factoring, fixing and reassembling (see \cite[\S3]{aHKRWproper-n}). So it remains for us to prove that the lower and upper quadrilaterals commute. 

For the lower quadrilateral, we need to look at the construction of the fixed-point bimodules (see \cite[\S3]{aHKRWproper-n} and our page~\pageref{defmorphism}). The imprimitivity bimodule $(X^{u|N},u^{G/N})$ is by definition the top right-hand corner of $(L(X)^{L(u)|N},L(u)^{G/N})$, and $\Fix(X^{u|N},u^{G/N})$ is by definition the top right-hand  corner of 
\[
\Fix(L(X^{u|N}),L(u^{G/N}))=\Fix(L(X)^{L(u)|N},L(u)^{G/N}).
\] 
Applying the first part of the present theorem to the object $(L(X),L(u),\phi_\K\oplus\psi)$ in $\Aa(G,(C_0(T),\rt))$ gives an isomorphism $\iota_{G,N}(L(X))$ of $\Fix_{G/N}(L(X)^{L(u)|N})$ onto $\Fix(L(X),L(u))$; this isomorphism is induced by the inclusion of $L(X)^{L(u)|N}$ in $M(L(X))$, which restricts on the corners to the inclusions of $\K^{\mu|N}$ and $B^{\beta|N}$ in $M(\K)$ and $M(B)$, and hence $\iota_{G,N}(L(X))$ restricts to $\iota_{G,N}(\K)$ and $\iota_{G,N}(B)$ on the diagonal corners. Now general nonsense (as in Lemma~4.6 of \cite{ER}, for example) implies that the restriction of $\iota_{G,N}(L(X))$ to the top right-hand corner is an isomorphism of $\Fix(X^{u|N},u^{G/N})\otimes_{\Fix B^{\beta|N}}\Fix(B,\beta)$ onto $\Fix(\K^{\mu|N})\otimes_{\Fix(\K,\mu)}\Fix(X,u)$. But the existence of such an isomorphism says precisely that the bottom quadrilateral commutes.

For the top quadrilateral we have to come to grips with  the property which characterises $\kappa_G$, and this means adding some $\iota$ maps. So we consider the diagram
\begin{equation*}\label{}
  \xymatrix{
    \Fix(A^{\alpha|N},\alpha^{G/N})
    \ar[rrrr]^-{\iota_{G,N}(A)}
    \ar[dd]_-{\overline{\kappa_N}}
    \ar[rd]^-{(\kappa_N)_{G/N}}
    &&&&
    \Fix(A,\alpha)
    \ar[ld]_-{\kappa_G}
    \ar[dd]^-{\overline{\kappa}}
    \\
    & \Fix(\K^{\mu|N},\mu^{G/N})
    \ar[ld]_-{\overline{\iota_{G/N}}}
    \ar[rr]^-{\iota_{G,N}(\K)}
    &&\Fix(\K,\mu)
    \ar[rd]^-{\overline{\iota_G}}
    \\
    M(K^{\mu|N})
    \ar[rrrr]^-{\overline{\iota_N}}
    &&&&
    M(\K).
  }
\end{equation*} 
Proposition~\ref{newkqr} implies that the right and left triangles commute. The bottom quadrilateral is the strictly continuous extension of the diagram
\begin{equation*}
    \xymatrix{
      \Fix(\K^{\mu|N},\mu^{G/N})
      \ar[rr]^-{\iota_{G,N}(\K)}
      \ar[d]_{\iota_{G/N}}
      &&
     \Fix(\K,\mu)
      \ar[d]^{\iota_G}
      \\
     M(\K^{\mu|N})
      \ar[rr]^-{\overline{\iota_N}}
      &&
     M(\K),
    }
  \end{equation*}
which commutes because the vertical arrows are inclusions and $\iota_{G,N}=\overline{\iota_N}$. Thus, since the $\iota$'s are all injections, it suffices to prove that the outside square commutes. Proposition~\ref{newkqr} tells us that $\overline{\iota_N}\circ \kappa_N$ is the restriction of $\overline{\kappa}$ to $A^{\alpha|N}\subset M(A)$, which we can rewrite as $\overline{\kappa}\circ\iota_N$. But 
\[
\overline{\iota_N}\circ \kappa_N=\overline{\kappa}\circ\iota_N
\Longrightarrow
\overline{\overline{\iota_N}\circ \kappa_N}=\overline{\overline{\kappa}\circ\iota_N}
\Longrightarrow
\overline{\iota_N}\circ \overline{\kappa_N}=\overline{\kappa}\circ\overline{\iota_N},
\]
which since $\iota_{G,N}=\overline{\iota_N}$ says that the outside square commutes, as required.
\end{proof}

\begin{thm}\label{cor-easyiis}
Suppose that $G$ acts freely and
  properly on a locally compact space $T$, $N$ is a closed
  normal subgroup of $G$, and $(A,\alpha,\phi)$ is an object in the
  semi-comma category $\Aa(G,(C_0(T),\rt))$. Then the following diagram commutes in $\CC$:
\begin{equation}
    \xymatrix{
      A\rtimes_{\alpha|,r} N
      \ar[rrrr]^-{Z(A,\alpha|N,\phi)}
      \ar[dd]_{A\rtimes_{\alpha,r} G}
      &&&&
     A^{\alpha|N}=\Fix(A,\alpha|N,\phi)
      \ar[d]^{Z(A^{\alpha|N}, \alpha^{G/N},\phi_N)}
      \\
      &&&&\Fix(A^{\alpha|N}, \alpha^{G/N},\phi_N)
     \ar[d]^{\iota_{G,N}}\\
     A\rtimes_{\alpha,r} G
      \ar[rrrr]^{Z(A,\alpha,\phi)}
      &&&&
      \Fix(A,\alpha,\phi).
    }
  \end{equation}
\end{thm}

\begin{proof} We will show that the map
 $\Omega: A_0\otimes_{E^N(A_0)} E^N(A_0)\to A_0$ defined by
$\Omega(a\otimes E^N(b))=aE^N(b)$
extends to a right-Hilbert $(A\rtimes_{\alpha|,r}N)$\,--\,$\Fix(A,\alpha,\phi)$-bimodule isomorphism of \[Z(A,\alpha|N,\phi)\otimes_{A^{\alpha|N}} Z(A^{\alpha|N},\alpha^{G/N},\phi_N)\]
onto $Z(A,\alpha,\phi)$. We begin by showing that $\Omega$ preserves the inner-product: for $a_1,a_2,b_1,b_2\in A_0$, we have
\begin{align*}
\langle a_1\otimes E^N(b_1)\,,\, a_2\otimes E^N&(b_2)\rangle_{\Fix(A,\alpha)}=\langle\langle a_2\,,\, a_1\rangle_{A^{\alpha|N}}E^N(b_1)\,,\, E^N(b_2)\rangle_{\Fix(A,\alpha)}\\
&=\iota_{G,N}\big(\langle E^N(a_2^*a_1)E^N(b_1)\,,\, E^N(b_2)\rangle_{\Fix(A^{\alpha|N},\alpha^{G/N})}\big)\\
&=\iota_{G,N}\big(E^{G/N}(  E^N(b_1)^*E^N(a_2^*a_1)^*E^N(b_2))\big)\\
&=\iota_{G,N}\big(E^{G/N}( E^N( E^N(b_1^*)a_1^*a_2E^N(b_2)))\big)\\
&=E^G( E^N(b_1^*)a_1^*a_2E^N(b_2))\quad\quad\text{(using \eqref{eqnforfixfix})}\\
&=\langle a_1E^N(b_1)\,,\, a_2E^N(b_2)\rangle_{\Fix(A,\alpha)}\\
&=\langle \Omega(a_1\otimes E^N(b_1))\,,\,\Omega( a_2\otimes E^N(b_2))\rangle_{\Fix(A,\alpha)}.
\end{align*}

Next, note that the left action of $z\in C_c(N,A)\subset A\rtimes_{\alpha|,r}N$ on $a\in A_0\subset Z(A,\alpha,\phi)$, which is given by
\[
z\cdot a=\int_G z(t)\alpha_t(a)\Delta_G(t)^{1/2}\, dt,
\]
is the integrated form of the covariant representation $(\pi,U)$, where
\[
\pi(a)b=ab\quad\text{and}\quad U_t(b)=\alpha_t(b)\Delta_G(t)^{1/2}\quad\text{for $a,b\in A_0$}.
\]
The inclusion of $A\rtimes_{\alpha|,r} N$ in $M(A\rtimes_{\alpha,r}G)$ is the integrated form of $(\id,i_G|)$, where $i_G: G\to M(A\rtimes_{\alpha,r}G)$ is the canonical map. Thus $w\in C_c(N,A)$ acts on $a\in A_0$ by
\begin{equation}\label{formleftact}
w\cdot a=\int_N g(n)\alpha_n(a)\Delta_N(n)^{1/2}\, dn
\end{equation}
(since $N$ is normal the modular functions of $G$ and $N$ coincide on $N$). From \eqref{formleftact}, we deduce, first, that $\Omega$ preserves the left action: for $a,b\in A_0$, we have
\begin{align*}
w\cdot \Omega(a\otimes E^N(b))
&=w\cdot(aE^N(b))=\int_N w(n)\alpha_n\big(aE^N(b) \big)\Delta_G(t)^{1/2}\, dn\\
&=\int_N w(n)\alpha_n (a)E^N(b)\Delta_G(t)^{1/2} \, dn\\
&=\Big(\int_N w(n)\alpha_n(a)\Delta_G(t)^{1/2}\, dn\Big) E^N(b)\\
&=\Omega(w\cdot a\otimes E^N(b))=\Omega\big(w\cdot( a\otimes E^N(b))\big).
\end{align*}

Next we use formula \eqref{formleftact} again to see that $\Omega$ has dense range, and hence is surjective. To do this, let $c\in A_0\subset Z(A,\alpha,\phi)$. Since $Z(A, \alpha|N,\phi)$ is an imprimitivity bimodule which is obtained by completing $A_0$, we can find an approximate identity $\{e_i\}$ for $A\rtimes_{\alpha|,r}N$ of the form $e_i=\sum_{j=1}^{n_i}{}_L\langle x_{i,j},x_{i,j}\rangle$, where all the elements $x_{i,j}$  belong to $A_0$ (by \cite[Lemma~6.3]{mw}, for example). Since the inclusion of $A\rtimes_{\alpha|,r}N$ in $M(A\rtimes_{\alpha,r}G)$ is nondegenerate, the net $\{e_i\}$ converges to $1$ strictly in $M(A\rtimes_{\alpha,r}G)$, and we can approximate $c$ in $Z(A,\alpha,\phi)$ by an element $e_i\cdot c$. But now we observe that the formula \eqref{formleftact} for the left action  of $w\in C_c(N,A)$ on $A_0\subset Z(A,\alpha,\phi)$ is exactly the same as the formula for the left action of $w$ on $A_0\subset Z(A,\alpha|N,\phi)$, and hence we have
\[
e_i\cdot c=\sum_{j=1}^{n_i}{}_L\langle x_{i,j},x_{i,j}\rangle\cdot c=\sum_{j=1}^{n_i}x_{i,j}\cdot \langle x_{i,j},c\rangle_{A^{\alpha|N}}=\sum_{j=1}^{n_i}x_{i,j}E^N(x_{i,j}^*c).
\]
Since each $x_{i,j}E^n(x_{i,j}^*c)$ belongs to the range of $\Omega$, we deduce that $\Omega$ has dense range. Now general nonsense tells us that $\Omega$ is a right-module homomorphism, and hence is an isomorphism of right-Hilbert bimodules.
\end{proof}

\section{Rieffel's Morita equivalence is equivariant}\label{sec:plums}

Let $[X,u]$ be a morphism from $(A,\alpha)$ to $(B,\beta)$ in    $\Aa(G)$. 
Recall that
$\Delta_{G,N}(s):=\Delta_G(s)\Delta_{G/N}(sN)^{-1}$
(see equation~\ref{Delta}).
By Lemma~3.23 of \cite{enchilada} there is a unique action $\alpha^\dec:G\to\Aut(A\rtimes_{\alpha|,r}N)$  such that
$
\alpha_s^\dec(w)(n)=\Delta_{G,N}(s)\alpha_s(w(s^{-1}ns))
$
for $w\in C_c(N,A)$.  The same formula defines an $\alpha^\dec$\,--\,$\beta^\dec$ compatible action $u^\dec$ on $C_c(N,X)\subset X\rtimes_{u|,r}N$.   By Theorem~3.24 of \cite{enchilada}, the assignments
\[(A,\alpha)\mapsto (A\rtimes_{\alpha|,r}N,\alpha^\dec)\quad\text{and}\quad (X,u)\mapsto (X\rtimes_{u|,r}N,u^\dec)
\] define a functor $\RCP_N^G$ from $\Aa(G)$ to $\Aa(G)$.

Composing with the quotient map $q:G\to G/N$ gives an action $\Inf(\alpha)=\alpha\circ q$ of $G$ on $A$, and there is an $\Inf(\alpha)$\,--\,$\Inf(\beta)$ compatible action $\Inf(u)=u\circ q$ on $X$.  By Corollary~3.18 of \cite{enchilada} the assignments 
\[(A,\alpha)\mapsto (A,\Inf(\alpha))\quad\text{and}\quad (X,u)\mapsto (X,\Inf(u))
\] define a functor $\Inf_{G/N}^G$ from $\Aa(G/N)$ to $\Aa(G)$.

Applying Theorem~3.5 of \cite{aHKRWproper-n} with $(T,G)=(G,N)$ gives a natural isomorphism between functors $\Fix_N$ and $\RCP_N$ defined on $\Aa(G,(C_0(G),\rt))$; the next theorem is an equivariant version of this natural isomorphism.  

\begin{thm}\label{thm-equivariantRieffel}
Suppose that $G$ is a locally compact group acting freely and
  properly on a locally compact space $T$, and that $N$ is a closed
  normal subgroup of $G$. Let $(A,\alpha,\phi)$ be an object in $\Aa(G,(C_0(T),\rt))$.
Then there is an $\alpha^\dec$\,--\,$\Inf(\alpha^{G/N})$ compatible  action $\Rie(\alpha)$ of $G$ on Rieffel's bimodule $Z(A,\alpha|N,\phi)$, and the assignment
\[
(A,\alpha)\mapsto (Z(A,\alpha|N,\phi),\Rie(\alpha))
\]
implements a natural isomorphism between the functors $\RCP_N^G$ and $\Inf_{G/N}^G\circ \Fix_N^{G/N}$.
\end{thm}

\begin{proof}  We start by showing that there is an $\alpha^\dec$\,--\,$\Inf(\alpha^{G/N})$ compatible action $\Rie(\alpha)$ on $Z(A,\alpha|,\phi)$ such that $\Rie(\alpha)_s(a)=\Delta_{G,N}(s)^{1/2}\alpha_s(a)$ for $a\in A_0$.

Fix $a,b\in A_0$. The modular functions of $G$ and $N$ agree on $N$ by normality, so
\begin{align*}
\alpha_s^\dec({}_{A\rtimes_{\alpha|,r}N} \langle a\,,\, b\rangle)(n)&=\Delta_{G,N}(s)\alpha_s\big( {}_{A\rtimes_{\alpha|,r}N} \langle a\,,\, b\rangle(s^{-1}ns) \big)\\
&=
\Delta_{G,N}(s)\alpha_s\big( a\alpha_{s^{-1}ns}(b)\Delta_N(s^{-1}ns)^{-1/2} \big)\\
&=\Delta_{G,N}(s)\alpha_s(a)\alpha_{ns}(b)\Delta_N(n)^{-1/2}
\\
&=\Delta_{G,N}(s) {}_{A\rtimes_{\alpha|,r}N} \langle \alpha_s(a)\,,\,\alpha_s( b)\rangle(n)
\\
&={}_{A\rtimes_{\alpha|,r}N} \langle \Rie(\alpha)_s(a)\,,\,\Rie(\alpha)_s( b)\rangle(n);
\end{align*}
replacing $a$ by $w\cdot a$ shows that $\Rie(\alpha)_s(w\cdot a)=\alpha_s^\dec(w)\cdot \Rie(\alpha)_s(a)$. Next, we let  $c\in A_0$, compute
\begin{align*}
c(\Inf\alpha^{G/N})_s(\langle a\,,\, b\rangle_{A^{\alpha|N}})
&=
c\alpha_s(\langle a\,,\, b\rangle_{A^{\alpha|N}})
=\alpha_s\big(\alpha_s^{-1}(c) E^N(a^*b)\big)\\
&=\alpha_s\left(\int_N\alpha_s^{-1}(c)\alpha_n(a^*b)\, dn  \right)
=\int_N c\alpha_{sn}(a^*b)\, dn \\
&=\Delta_{G,N}(s)\int_N c\alpha_{ns}(a^*b)\, dn \quad\text{(see \eqref{Delta})}\\
&=\Delta_{G,N}(s)\int_N c\alpha_{n}(\alpha_s(a)^*\alpha_s(b))\, dn \\
&=cE^N\big(\Rie(\alpha)_s(a)^*\Rie(\alpha)_s(b)\big)\\
&=c\langle\Rie(\alpha)_s(a)\,,\,\Rie(\alpha)_s(b)\rangle_{A^{\alpha|N}},
\end{align*}
and replacing $b$ by $b\cdot l$ shows that 
$\Rie(\alpha)_s(a\cdot l)=\Rie(\alpha)_s(a)\cdot  (\Inf\alpha^{G/N})_s(l)$ for $l\in E^N(A_0)$. In particular, 
\[
\|\Rie(\alpha)_s(a)\|^2
=\|\alpha_s^\dec({}_{A\rtimes_{\alpha|,r}N}  \langle a\,,\, a \rangle) \|
=\|{}_{A\rtimes_{\alpha|,r}N}   \langle a\,,\, a \rangle) \|
=\|a\|^2,
\]
so  $\Rie(\alpha)_s$ extends to an automorphism of   $Z(A,\alpha|N,\phi)$.

To see that  $\Rie(\alpha)$ is strongly continuous it suffices to see that  for a fixed $a\in A_0$, $s\mapsto {}_{A\rtimes_{\alpha|,r}N}   \langle \Rie(\alpha)_s(a)\,,\, a\rangle={}_{A\rtimes_{\alpha|,r}N}   \langle \Delta_{G,N}(s)\alpha_s(a)\,,\, a\rangle$ is continuous  (see, for example, \cite[Lemma~2.1]{aHRWproper2}).  So the strong continuity of $\Rie(\alpha)$ follows from the strong continuity of $\alpha$ and the continuity of $\Delta_{G,N}$.

Let
$[X,u]: (A,\alpha,\phi)\to (B,\beta,\psi)$ be a morphism in
$\Aa(G,(C_0(T),\rt))$. We factor $X$ using \cite[Corollary~3.5]{aHKRWproper-n}:
\begin{equation*}\label{eq-main1}
  \xymatrix{
    (A,\alpha,\phi)
    \ar[r]^-{\kappa}
    &(\K(X),\mu,\phi_{\K})
    \ar[r]^-{(X,u)}
    &(B,\beta,\psi).
  }
\end{equation*}
We need to show that the outer square of the
following diagram
\begin{equation*}\label{eq-main-2}
  \xymatrix{
    (A\rtimes_{\alpha|,r} N,\alpha^\dec)
    \ar[rrrr]^-{(Z(A,\alpha|N,\phi),\Rie(\alpha))}
    \ar[dd]_-{(X\rtimes_{r}N,u^\dec)}
    \ar[rd]^-{\kappa\rtimes_r N}
    &&&&
    (A^{\alpha|N},\Inf(\alpha^{G/N}))
    \ar[ld]_-{\kappa_N}
    \ar[dd]^-{(\Fix(X,u|),\Inf(u^{G/N}))}
    \\
    & \K\rtimes_{\mu|,r} N
    \ar[ld]^-{\ (X\rtimes_{u|,r} N,\mu^\dec)}
    \ar[rr]^-{Z(\K,\mu|N,\phi_{\K})}
    &&\K^{\mu|N}
    \ar[rd]_-{(X^{u|},\Inf(u^{G/N}))\ \ }
    \\
    (B\rtimes_{\beta|,r}N,\beta^\dec)
    \ar[rrrr]_-{(Z(B,\beta|N,\psi),\Rie(\beta))}
    &&&&
    (B^{\beta|N},\Inf(\beta^{G/N}))
  }
\end{equation*}
commutes equivariantly. Without the actions the diagram commutes by \cite[Theorem~3.5]{aHKRWproper-n}. The left and right triangles commute equivariantly because $\RCP_N^G$ is a functor from $\Aa(G)$ to $\Aa(G)$ and  $\Inf_{G/N}^G\circ\Fix_N^{G/N}$ is a functor from  $\Aa(G,(C_0(T),\rt))$ to $\Aa(G)$.

By \cite[Theorem~3.2]{kqrproper}, the isomorphism representing the upper quadrilateral is
\[
\Phi:Z(A,\alpha|N,\phi)\otimes_{A^{\alpha|}}\K^{\mu|N}\to Z(\K,\mu|N,\phi_\K)\ \text{ given by }\ \Phi(a\otimes E^N(l))=\kappa(a)E^N(l)
\]
 for $a\in A_0$ and $l\in (\K^{\mu|N})_0$. The $\alpha$\,--\,$\mu$ equivariance of $\kappa$ gives the equivariance of $\Phi$:
\begin{align*}
\Phi(\Rie(\alpha)_s(a)\otimes\Inf(\mu)_s(E^N(l)))
&=\kappa\big(\Rie(\alpha)_s(a)\big)\Inf(\mu)_s(E^N(l))\\
&=\Delta_{G,N}(s)^{1/2}\kappa(\alpha_s(a))\mu_s(E^N(l))\\
&=\Delta_{G,N}(s)^{1/2}\mu_s(\kappa(a)E^N(l))\\
&=\Rie(\mu)_s(\Phi(a\otimes E^N(l))).
\end{align*}
So the upper quadrilateral commutes equivariantly, and it remains to show that the lower quadrilateral commutes equivariantly.

The imprimitivity bimodule $W:=Z(L(X),L(u)|N,\phi_{\K}\oplus\psi)$ carries a $L(u)^\dec$\,--\,$\Inf(L(u)^{G/N})$ compatible action $\Rie(L(u))$, and since the isomorphism of $L(X)\rtimes_{L(u)|,r}N$ onto $L(X\rtimes_{u|,r}N)$ is $L(u)^\dec$\,--\,$L(u^\dec)$ equivariant and $L(X)^{L(u)|}=L(X^{u|})$, we can view  $(W,\Rie(L(u)))$ as a $(L(X\rtimes_{u|,r}N), L(u^\dec))$\,--\,$(L(X^{u|}),\Inf L(u^{G/N}))$ imprimitivity bimodule.  Let $p,q\in M(L(X\rtimes_{u|,r}N))$  and $p',q'\in M(L(X^{u|}))$ be the complementary corner projections.

Let $k\in\K_0$. The map $k\mapsto \big(\begin{smallmatrix} k&0\\0&0\end{smallmatrix}\big)$ extends to a norm-preserving map of $\K_0\subset Z(\K,\mu|N,\phi_\K)$ into $W$, and hence extends to an isomorphism of $Z(\K,\mu|N,\phi_\K)$ onto $pFp'$ which is compatible with the inclusion of $\K^{\mu|N}$ onto $pL(X)^{L(u)|N}p'$.  Similarly, $Z(B,\beta|N, \psi)$ is isomorphic to $qWq'$.

In \cite[Theorem~3.5]{aHKRWproper-n} we showed that the lower quadrilateral commutes by applying \cite[Lemma~4.6]{ER} to get isomorphisms
\[
\Psi_1:pWp'\otimes_{\K^{\mu|}} X^{u|}\to pWq'\ \text{ given by }\ \Psi_1(k\otimes m)=k\cdot m,
\]
where the action of $m\in X^{u|}$ on $k\in pWp'$ is the right action of $L(X^{u|})$ on $W$, so that $k\cdot m$ is the product of
\[\begin{pmatrix}k&0\\0&0\end{pmatrix}\in L(X_0)\quad{\text{and}}\quad \begin{pmatrix}0&m\\0&0\end{pmatrix}\in M(L(X)),
\]
and
\[\Psi_2:X\rtimes_{u|,r}N\otimes_{B\rtimes_{\beta|,r}N} qWq'\to pWq'\ \text{ given by }\ \Psi_2(z\otimes b)=z\cdot b.\]
The actions $\Rie(\mu)$ and $\Rie(u)$ of $G$ on $pWp'=Z(\K,\mu|,\phi_\K)$ and $pWq'$  are  restrictions  of the action $\Rie(L(u))$ of $G$ on $W$. The action $\Inf(u^{G/N})$ of $G$ on $X^{u|}$ is the restriction to the top-right corner of the action $\Inf(L(u)^{G/N})$ of $G$ on $L(X)^{L(u)|}$.
Since $\Rie(L(u))$ and $\Inf(L(u)^{G/N})$ are compatible actions,  $\Psi_1$ is $(\Rie(\mu)\otimes\Inf(u^{G/N}))$\,--\,$\Rie(u)$ equivariant.
Similarly,  $\Psi_2$ is equivariant.
Thus the lower quadrilateral commutes equivariantly, as required.
\end{proof}

Let $\delta:B\to M(B\otimes C^*(G))$ be a normal nondegenerate coaction.  When $(T,G,N)=(G,G,N)$ and $(A,\alpha,\phi)=(B\rtimes_\delta G,\hat\delta, j_G)$, Theorem~\ref{thm-equivariantRieffel} is part of \cite[Theorem~4.3]{enchilada} (it does not include the assertions about equivariance with respect to coactions). 
This application was our motivation for formulating Theorem~\ref{thm-equivariantRieffel}.

\section{Naturality for crossed products by coactions}\label{sec:nat}

Throughout this section  $H\subset K$ are closed subgroups of $G$ such that $H$ is normal in $K$. We begin by showing that under these circumstances the functor $\RCP_{G/H}$ of \cite[Proposition~5.5]{aHKRWproper-n} has an equivariant version.

\begin{lemma}\label{lem-functorscoincide}
For every object $(B,\delta)$ and every morphism $[X,\Delta]$ in $\Aca^\nor(G)$, there are actions $(\hat\delta|H)^{K/H}$ and $(\hat\Delta|H)^{K/H}$ of $K/H$ on $B\rtimes_{\delta,r}(G/H)$ and $X\rtimes_{\Delta,r}(G/H)$ such that 
\begin{align}
(\hat\delta|H)^{K/H}_{tH}(j_B(b)j_G|(f))&=j_B(b)j_G|(\rt_{tH}(f))\ \text{ for $b\in B$, $f\in C_0(G/H)$, and }\label{actonG/H1}\\
(\hat\Delta|H)^{K/H}_{tH}(j_X(x)j_G|(f))&=j_G(x)j_G|(\rt_{tH}(f))\ \text{ for $x\in X$, $f\in C_0(G/H)$.}\label{actonG/H2}
\end{align}
 Adding these actions makes $\RCP_{G/H}$ into a functor $\RCP_{G/H}^{K/H}$ from $\Aca^\nor(G)$ to $\Aa(K/H)$, and this functor coincides with $\Fix_H^{K/H}\circ \CP_K$.
\end{lemma}

\begin{proof}
Proposition~\ref{prop-actG/N} shows that the functor $\Fix_H:\Aa(G, (C_0(G),\rt))\to\CC$ extends to a functor $\Fix_H^{K/H}:\Aa(G, (C_0(G),\rt))\to\Aa(K/H)$ , and hence $\Fix_H\circ\CP_K$ also extends. Since the effect of the two functors $\RCP_{G/H}$ and $\Fix_H\circ\CP_K$ is exactly the same (they yield the same subspaces of multiplier algebras and bimodules), we just need to check that the actions have the effect described in \eqref{actonG/H1} and \eqref{actonG/H2}. But the automorphism $(\hat\delta|H)^{K/H}_{tH}$ of $\Fix (B\rtimes_{\delta}G,\hat\delta|H)$ is by definition the restriction of $(\hat\delta|H)_t=\hat\delta_t$, so this follows immediately from the definition of the dual action.
\end{proof}

\begin{thm}\label{newMe} Suppose that $H$ and $K$ are closed subgroups 
of a locally compact group $G$ such that $H$ is normal in
$K$. Let $(B,\delta)$ be an object in $\Aca^\nor(G)$, and let 
\[
\iota_{K,H}: \Fix(B\rtimes_{\delta,r}(G/H),\hat\delta^{K/H}, 
(j_G)_H)\to \Fix(B\rtimes_{\delta}G, \hat\delta|K,j_G)
\]
be the isomorphism of Theorem~\ref{thm-newfix}. Then
\begin{equation}\label{Rieffbimodswithiota}
(B,\delta)\mapsto [\iota_{K,H}]\circ [Z\big(B\rtimes_{\delta,r} (G/H),\hat\delta^{K/H}, 
(j_G)_H\big)]
\end{equation}
is a natural isomorphism between the 
functors
\begin{align*}
&\RCP_{K/H}\circ \RCP_{G/H}^{K/H}:(B,\delta)\mapsto (B\rtimes_{\delta,r} 
(G/H))\rtimes_{\hat\delta^{K/H},r}(K/H), \text{ and}\\
&\RCP_{G/K}:(B,\delta)\mapsto B\rtimes_{\delta,r} (G/K)
\end{align*}
from $\Aca^\nor(G)$ to $\CC$.
\end{thm}

\begin{proof}
With  $(T,G)=(G/H,K/H)$  
and $(A,\alpha,\phi)=(B\rtimes_{\delta,r}(G/H),\hat\delta^{K/H}, (j_G)_H)$,
Theorem~3.5 of \cite{aHKRWproper-n} 
says that the Rieffel bimodules $Z(B\rtimes_{\delta,r} 
(G/H),\hat\delta^{K/H}, (j_G)_H)$ give  natural isomorphisms between
\begin{gather*}
 \RCP_{K/H}:\Aa(K/H,(C_0(G/H),\rt))\to \CC,\ \text{ and}\\
  \Fix_{K/H}:\Aa(K/H,(C_0(G/H),\rt))\to \CC.
\end{gather*}
Thus the Rieffel bimodules also give natural isomorphisms between
\[\RCP_{K/H}\circ\Fix_H^{K/H}\circ\CP_K\quad\text{and}\quad 
\Fix_{K/H}\circ\Fix_H^{K/H}\circ \CP_K\]
from $\Aca^\nor(G)$ to $\CC$.
But $\Fix_H^{K/H}\circ \CP_K=\RCP_{G/H}^{K/H}$ by  
Lemma~\ref{lem-functorscoincide}, and Theorem~\ref{thm-newfix} says that $\iota_{K,H}$ is a natural isomorphism of $ \Fix_{K/H}\circ\Fix_H^{K/H}$ 
onto $\Fix_K$, so the composition  \eqref{Rieffbimodswithiota} is a natural 
isomorphism from
\[\RCP_{K/H}\circ\RCP_{G/H}^{K/H}\quad\text{to}\quad \Fix_K\circ\CP_K
\]
Another application of \cite[Proposition~5.5]{aHKRWproper-n} says that   $\Fix_K\circ\CP_K =\RCP_{G/K}$, and hence we have proved the theorem.
\end{proof}

\section{Induction of representations between crossed products by homogeneous spaces}

Suppose that $\delta$ is a normal coaction of a locally compact group $G$ on a $C^*$-algebra $B$, and $H$, $K$ are closed subgroups of $G$ with $H\subset K$. We want to define an induction process $\Ind_{G/K}^{G/H}$ from $\Rep\big( B\rtimes_{\delta,r}(G/K)\big)$ to $\Rep \big(B\rtimes_{\delta,r}(G/H)\big)$ which has ``the usual properties one would expect,'' and fortunately we have previously described what we think this means for actions. So, dualising the criteria described in the introduction to \cite{aHKRW-JFA}, we want:
\begin{enumerate}
\item an \emph{imprimitivity theorem} which characterises induced 
representations;

\smallskip
\item \emph{regularity}: the representations induced from the trivial quotient $\{G/G\}$ 
are the regular representations; and 

\smallskip
\item  \emph{induction-in-stages}: for $H\subset K\subset L$, we have $\Ind_{G/K}^{G/H}\circ \Ind_{G/L}^{G/K}=\Ind_{G/L}^{G/H}$.
\end{enumerate}

When all the subgroups are normal, the process constructed by Mansfield in \cite{man} has these properties: Mansfield himself proved an imprimitivity theorem \cite[Theorem~28]{man} and that the representations induced from $G/G$ are the regular representations \cite[Proposition~21]{man}, and induction-in-stages was later proved in \cite[Corollary~4.2]{kqrold}. (There are potentially confusing switches of hypotheses between \cite{man} and \cite{kqrold}, but since one can pass from normal to reduced coactions and back again \cite{quigg-fr}, the results all hold for the reduced coactions used in \cite{man} and the normal coactions used in \cite{kqrold}. The analogous results for maximal coactions were only recently obtained in \cite{aHKRW-JFA}.)

Our goal here is to do what we can for non-normal subgroups. We have not been able to completely eliminate normality hypotheses, but we get a satisfactory theory for pairs of subgroups $H$, $K$ such that $H$ is normal in $K$. For such a pair, Theorem~\ref{newMe} gives a Morita equivalence $Z(B\rtimes_{\delta,r} (G/H),\hat\delta^{K/H}, (j_G)_H)$ between $(B\rtimes_{\delta,r}(G/H))\rtimes_{\widehat\delta^{K/H}}(K/H)$ and the fixed-point algebra $\Fix_{K/H}(B\rtimes_{\delta,r}(G/H),\widehat\delta^{K/H},(j_G)_H)$. Since $B\rtimes_{\delta,r}(G/H)$ is itself the fixed-point algebra for the system $(B\rtimes_\delta G,\widehat\delta|_H,j_G)$ and $H$ is normal in $K$, we can use the isomorphism
\[
\iota_{K,H}:\Fix_{K/H}\circ\Fix_{H}^{K/H}(B\rtimes_\delta G,\widehat\delta|_K,j_G)\to \Fix_{K}(B\rtimes_\delta G,\widehat\delta|_K,j_G)=B\rtimes_{\delta,r}(G/K)
\]
of Theorem~\ref{thm-newfix} to view $Z(B\rtimes_{\delta,r} (G/H),\widehat\delta^{K/H}, (j_G)_H)$ as a right Hilbert $(B\rtimes_{\delta,r}(G/K))$-module, and throw away the left action of $K/H$ to obtain a right-Hilbert $(B\rtimes_{\delta,r}(G/H))$--$(B\rtimes_{\delta,r}(G/K))$ bimodule $Z_{G/K}^{G/H}(B,\delta)$; formally,  $Z_{G/K}^{G/H}(B,\delta)$ is the composition
\begin{equation*}
    \xymatrix{
 B\rtimes_{\delta,r}(G/H)
      \ar[rrr]^-{Z(B\rtimes_{\delta,r} (G/H),\widehat\delta^{K/H})}
      &&&
     \Fix_{K/H}(B\rtimes_{\delta,r} (G/H),\widehat\delta^{K/H})
     \ar[r]^-{\iota_{K,H}}
      &B\rtimes_{\delta,r}(G/K).
          }
  \end{equation*}
We now define
\[
\Ind_{G/K}^{G/H}:=Z_{G/K}^{G/H}(B,\delta)\dashind.
\]

\begin{remark}
Because this induction process is defined using Rieffel induction as in \cite[\S2.4]{tfb}, it automatically has lots of nice properties: it is functorial on representations and intertwining operators; it preserves weak containment, and hence is well-defined on ideals \cite[Corollary~2.73]{tfb}; and it is continuous \cite[Corollary~3.35]{tfb}. Further, the naturality of Theorem~\ref{newMe} implies that $\Ind_{G/K}^{G/H}$ will be compatible with the Rieffel-induction processes arising from morphisms $[W,\Delta]$ in the category $\Aca^\nor(G)$. 
\end{remark}

Adapting the argument of \cite[Proposition~2.1]{aHKRW-JOT} to reduced crossed products  gives an imprimitivity theorem for this process:

\begin{cor}  Suppose that $H$ and $K$ are closed subgroups of a locally compact group $G$ such that $H$ is normal in $K$, that $\delta$ is a normal coaction of $G$ on $B$, and that $\pi$ is a nondegenerate representation of $B\rtimes_{\delta,r}(G/H)$ on $\H_\pi$.  Then there is a nondegenerate representation $\tau$ of $B\rtimes_{\delta,r}(G/K)$ such that $\pi$ is unitarily equivalent to $\Ind_{B\rtimes_{\delta,r}(G/K)}^{B\rtimes_{\delta,r}(G/H)}\tau$ if and only if there exists a unitary representation $U$ of $K/H$ on $\H_\pi$ such that $(\pi,U)$ is covariant for $\delta^{K/H}$ and $\pi\rtimes U$ factors through the reduced crossed product $(B\rtimes_{\delta, r}(G/H))\rtimes_{\delta^{K/H},r}(K/H)$. (The last requirement is automatic if $K/H$ is amenable).
\end{cor}

Next, suppose that $\pi$ is a representation of $B=B\rtimes_{\delta|}(G/G)$. Since $\delta$ is normal, the reduction $\delta^r$ is a reduced coaction of $G$ on $B$ (as opposed to a quotient of $B$), and $(B\rtimes_\delta G,j_B,j_G)$ is also the crossed product by the reduction $\delta^r$ (by Propositions~3.3 and~2.8 of \cite{quigg-fr} or Theorem~4.1 of \cite{rae}). Thus it follows from the discusssion preceding Theorem~6.2 of \cite{kqrproper} that $Z_{G/G}^G(B,\delta)$ is the bimodule $\overline{\DD}$ used in \cite{man}, and Proposition~21 of \cite{man} implies that $\Ind_{G/G}^G\pi$ is equivalent to the regular representation
$((\pi\otimes\lambda)\circ\delta,1\otimes M)$ on $\H_\pi\otimes L^2(G)$.

So our induction process has properties (1) and (2). Property (3) is harder.

\begin{thm}
Suppose  that $\delta$ is a normal coaction of $G$ on $B$, and that $H$, $K$ and $L$ are closed subgroups of $G$ such that $H\subset K\subset L$ and both $H$ and $K$ are normal in  $L$. Then for every representation $\pi$ of $B\rtimes_{\delta,r}(G/L)$,  $\Ind_{G/K}^{G/H}\big(\Ind_{G/L}^{G/K}\pi\big)$ is unitarily equivalent to $\Ind_{G/L}^{G/H}\pi$.
\end{thm}

\begin{proof}
We will prove the stronger statement that 
\begin{equation}\label{iisiso}
Z_{G/K}^{G/H}(B,\delta)\otimes_{B\rtimes(G/K)} Z_{G/L}^{G/K}(B,\delta)\cong Z_{G/L}^{G/H}(B,\delta)
\end{equation}
as right-Hilbert $(B\rtimes_{\delta,r}(G/H))$--$(B\rtimes_{\delta,r}(G/L))$ bimodules.

To make our calculations less cluttered, we write $(C,\beta)$ for $(B\rtimes_\delta G,\widehat\delta)$, and write, for example, $\Fix_HC$ for $\Fix(C,\beta_H,j_G)$. Now we apply Theorem~\ref{cor-easyiis} with $T=G$, $(G,N)=(L/H,K/H)$ and $(A,\alpha,\phi)=(\Fix_HC,\beta^{L/H},j_G)$, obtaining the following commutative diagram (in which we have omitted the crossed products on the left because they are not relevant for our purposes):
\begin{equation*}
    \xymatrix{
      \Fix_HC
      \ar[rrr]^{Z(\Fix_HC,(\beta^{L/H})|_{K/H})}
      \ar[dd]_{=}
      &&&
     \Fix_{K/H}(\Fix_HC,\beta^{L/H}|_{K/H})
      \ar[d]^{Z(\Fix_{K/H}(\Fix_HC,\beta^{L/H}|_{K/H}))}
      \\
      &&&\Fix_{L/K}\big(\Fix_{K/H}(\Fix_HC,\beta^{L/H}|_{K/H}),(\beta^{L/H})^{(L/H)/(K/H)}\big)
     \ar[d]^{\iota_{L/H,K/H}}\\
     \Fix_HC
      \ar[rrr]^{Z(\Fix_HC,\beta^{L/H})}
      &&&
      \Fix_{L/H}(\Fix_HC,\beta^{L/H})
    }
  \end{equation*}
To get a diagram involving the modules $Z_{G/K}^{G/H}(B,\delta)$ which were used to define our induction process, we need to add lots of ``fix--fix'' isomorphisms. This gives us a diagram of the following shape:
\begin{equation*}
    \xymatrix{
      \Fix_HC
      \ar[rr]^{Z(\Fix_HC)}
      \ar[dd]_{=}
      &&
     \Fix_{K/H}\Fix_HC
      \ar[d]_{Z(\Fix_{K/H}\Fix_HC)}\ar[r]^{\iota_{K,H}}
      &\Fix_KC\ar[d]^{Z(\Fix_KC,\beta^{L/K})}
      \\
      &&\Fix_{(L/H)/(K/H)}\Fix_{K/H}\Fix_HC
           \ar[d]_{\iota_{L/H,K/H}}\ar[r]^{\qquad\quad\theta}
           &\Fix_{L/K}\Fix_KC\ar[d]^{\iota_{L,K}}\\
     \Fix_HC
      \ar[rr]^{Z(\Fix_HC)}
      &&
      \Fix_{L/H}\Fix_HC\ar[r]^{\iota_{L,H}}&\Fix_{L}C=B\rtimes (G/L)
    }
  \end{equation*}
Notice that the compositions along the top, right and bottom are the three bimodules $Z_{G/K}^{G/H}(B,\delta)$, $Z_{G/L}^{G/K}(B,\delta)$ and  $Z_{G/L}^{G/H}(B,\delta)$, so to prove \eqref{iisiso}, it will suffice to prove that this diagram commutes in $\CC$. Theorem~\ref{cor-easyiis} implies that the left-hand rectangle commutes.
The square in the top-right comes from the isomorphism $\iota_{K,H}$ of $\Fix_{K/H}\Fix_HC$ onto $\Fix_KC$, which carries the action $(\beta^{L/H})^{(L/H)/(L/K)}$ of $L/K=(L/H)/(L/K)$ into $\beta^{L/K}$, and hence induces an isomorphism $\theta$ of fixed-point algebras, which appears in the middle row. So the top right-hand square commutes. It remains to show that the bottom right-hand square commutes, and this will require us to look closely at the construction of the maps.

First of all, the top arrow $\theta$ is induced by the isomorphism $\iota_{K,H}$, and since the fixed-point algebras are by definition subalgebras of $M(\Fix_{K/H}\Fix_HC)$ and $M(\Fix_KC)$, $\theta$ is simply the strictly continuous extension $\overline{\iota_{K,H}}$ to the multiplier algebra. The vertical arrow $\iota_{L,K}$ is the strictly continuous extension of the inclusion $\iota_K$ of $\Fix_KC$ in $M(C)$, which is nondegenerate by Corollary~\ref{fixKQR}. Thus \begin{equation}\label{ES}
\iota_{L,K}\circ\theta=\overline{\iota_{K}}\circ\overline{\iota_{K,H}}=\overline{\overline{\iota_K}\circ\iota_{K,H}}.
\end{equation}
For $m\in \Fix_{K/H}\Fix_HC$, $\iota_{K,H}(m)$ belongs to $\Fix_KC$, and hence $\overline{\iota_K}\circ\iota_{K,H}(m)=\iota_K(\iota_{K,H}(m))$; since $\iota_K$ is the inclusion, we have $\overline{\iota_K}\circ\iota_{K,H}=\iota_{K,H}=\overline{\iota_H}$.

On the other hand, $\iota_{L/H,L/K}$ is the strictly continuous extension of the inclusion $j$ of $\Fix_{K/H}\Fix_HC$ in $M(\Fix_HC)$, so
\begin{equation}\label{SE}
\iota_{L,H}\circ\iota_{L/H,L/K}=\overline{\iota_H}\circ\overline{j}=\overline{\overline{\iota_H}\circ j}.
\end{equation}
Since $j$ is an inclusion, $\overline{\iota_H}\circ j=\overline{\iota_H}$; since  we have already seen that $\overline{\iota_H}=\overline{\iota_K}\circ\iota_{K,H}$, Equations~\ref{ES} and~\ref{SE} imply that $\iota_{K,H}\circ\theta$ and $\iota_{L,H}\circ\iota_{L/H,L/K}$ are the strictly continuous extensions of the same homomorphism from $\Fix_{K/H}\Fix_HC$ to $M(C)$, and hence are equal. Thus the diagram commutes, as claimed, and this completes the proof. 
\end{proof}

\end{document}